\documentclass[12pt,fullpage,doublespace]{article}
\usepackage{graphicx} 
\usepackage{setspace} 
\usepackage{latexsym} 
\usepackage{amsfonts} 
\usepackage{amsmath} 
\usepackage{amssymb}
\usepackage{accents}
\usepackage{textcomp}
\usepackage{xcolor}
\usepackage{enumerate}
\usepackage{bm}
\usepackage[margin=1in]{geometry}
\usepackage{amsthm}
\usepackage{amssymb,latexsym,amsmath}
\usepackage{graphics}
\usepackage{tikz}
\usepackage{changepage}
\usepackage{comment}
\usepackage{mathrsfs}
\usepackage{titlesec}
\usepackage{lipsum}
\usetikzlibrary{matrix,arrows,positioning,scopes}
\usepackage[colorlinks=true,linkcolor=blue]{hyperref}

\newtheorem{theorem}{Theorem}[section]
\newtheorem{conjecture}[theorem]{Conjecture}
\newtheorem{definition}[theorem]{Definition}
\newtheorem{question}[theorem]{Question}

\newtheorem{lemma}[theorem]{Lemma}

\newtheorem{corollary}[theorem]{Corollary}
\newtheorem{remark}[theorem]{Remark}
\newtheorem{claim}[theorem]{Claim}

\numberwithin{equation}{section}

\newcommand{\rest}{\restriction}

\newcommand{\dom}{{\rm dom}}

\newcommand{\lh}{{\rm lh}}

\def\a{\alpha}

\def\l{\lambda}

\renewcommand{\models}{\vDash}
\newcommand{\powerset}{{\wp}}

\def\P{{\mathcal{P} }}

\def\Q{{\mathcal{ Q}}}

\def\L{{\rm{L}}}

\def\R{{\mathcal R}}

\def\M{{\mathcal{M}}}
\def\N{{\mathcal{N}}}
\def\F{{\mathcal{F}}}
\def\T {{\mathcal{T}}}
\def\U{{\mathcal{U}}}

\def\F{{\mathcal{F}}}

\def\VT{{\vec{\mathcal{T}}}}

\newcommand{\cA}{\mathcal{A}}

\DeclareMathOperator{\Env}{Env}
\DeclareMathOperator{\bfEnv}{\bf{Env}}
\DeclareMathOperator{\Lp}{Lp}
\DeclareMathOperator{\Ord}{Ord}
\DeclareMathOperator{\OD}{OD}
\DeclareMathOperator{\bfDelta}{\bf{\Delta}}
\DeclareMathOperator{\bfGamma}{\bf{\Gamma}}

\def\a{\alpha}

\def\l{\lambda}

\def\P{{\mathcal{P} }}

\def\Q{{\mathcal{ Q}}}

\def\L{{\rm{L}}}

\def\R{{\mathcal R}}

\def\M{{\mathcal{M}}}
\def\N{{\mathcal{N}}}
\def\F{{\mathcal{F}}}
\def\T {{\mathcal{T}}}
\def\U{{\mathcal{U}}}

\def\G{{\mathcal{G}}}

\def\VT{{\vec{\mathcal{T}}}}

\newcommand{\ins}{\trianglelefteq}

\newcommand{\pins}{\triangleleft}

\newcommand{\crit}{\mathrm{crit}}
\newcommand{\RR}{\mathbb{R}}

\newcommand{\OR}{\text{OR}}
\newcommand{\J}{\mathcal J}
\newcommand{\un}{\cup}
\newcommand{\sub}{\subseteq}
\newcommand{\om}{\omega}
\newcommand{\ZF}{\mathsf{ZF}}
\newcommand{\AD}{\mathsf{AD}}
\newcommand{\ZFC}{\mathsf{ZFC}}

\newcommand{\es}{\mathbb{E}}
\newcommand{\Tt}{\mathcal{T}}

\newcommand{\sats}{\vDash}
\newcommand{\cut}{\backslash}
\newcommand{\Ult}{\mathrm{Ult}}

\newcommand{\inter}{\cap}

\newcommand{\BB}{\mathbb{B}}
\newcommand{\Coll}{\mathrm{Col}}

\renewcommand{\l}{\mathit{l}}

\renewcommand{\OR}{\textit{o}}
\newcommand{\nth}{\text{th}}

\newcommand{\HC}{\mathrm{HC}}

\newcommand{\BBB}{\mathfrak{B}}
\newcommand{\Ll}{\mathcal{L}}
\newcommand{\all}{\forall}
\newcommand{\ex}{\exists}

\newcommand{\Mbar}{\bar{\M}}

\newcommand{\MC}{\mathsf{MC}}

\newcommand{\MFsharp}{\mathfrak{M}}
\newcommand{\Jop}{\J^{\mathrm{m}}}

\newcommand{\Fop}{\mathcal{F}}

\newcommand{\g}{\mathrm{g}}

\newcommand{\param}{\mathfrak{P}}
\newcommand{\gTheta}{\mathsf{G}}

\newcommand{\Los}{\L{}o\'{s}}
\DeclareMathOperator{\HOD}{HOD}
\DeclareMathOperator{\Code}{Code}
\DeclareMathOperator{\Con}{Con}

\titleformat{\section}{\normalsize\centering}{\thesection.}{1em}{}
\titleformat{\subsection}{\normalsize\centering}{\thesubsection.}{1em}{}
\titleformat{\subsubsection}{\normalsize}{\thesubsubsection.}{1em}{}
\numberwithin{equation}{section}

\newcommand*{\TitleFont}{%
      \usefont{\encodingdefault}{\rmdefault}{b}{n}%
      \fontsize{12}{16}%
      \selectfont}

 \input xy
 \xyoption{all}
\begin{document}
\title{\TitleFont DETERMINACY FROM STRONG COMPACTNESS OF $\omega_1$}
\renewcommand{\thefootnote}{\fnsymbol{footnote}} 
\date{\today}
\renewcommand{\thefootnote}{\arabic{footnote}}
\author{\fontsize{11}{13}NAM TRANG\footnote{UC Irvine, Irvine, CA, USA. Email: ntrang@math.uci.edu}\\
TREVOR WILSON\footnote{Miami University, Oxford, Ohio, USA. Email: twilson@miamioh.edu}}
\maketitle
\begin{abstract}
In the absence of the Axiom of Choice, the ``small'' cardinal $\omega_1$ can exhibit properties more usually associated with large cardinals, such as strong compactness and supercompactness.
For a local version of strong compactness, we say that $\omega_1$ is $X$-strongly compact (where $X$ is any set) if there is a fine, countably complete measure on $\powerset_{\omega_1}(X)$.
Working in $\mathsf{ZF} + \mathsf{DC}$, we prove that the
$\powerset(\omega_1)$-strong compactness and $\powerset(\mathbb{R})$-strong compactness of $\omega_1$ are equiconsistent with $\mathsf{AD}$ and $\mathsf{AD}_\mathbb{R} + \mathsf{DC}$ respectively, where $\mathsf{AD}$ denotes the Axiom of Determinacy and $\mathsf{AD}_\mathbb{R}$ denotes the Axiom of Real Determinacy.
The $\powerset(\mathbb{R})$-supercompactness of 
$\omega_1$ is shown to be slightly stronger than $\mathsf{AD}_\mathbb{R} + \mathsf{DC}$, but its consistency strength is not computed precisely.
An equiconsistency result at the level of $\mathsf{AD}_\mathbb{R}$ without $\mathsf{DC}$ is also obtained.
\end{abstract}

\newpage
\tableofcontents
\newpage

\thispagestyle{empty}
\thispagestyle{empty}
\section{INTRODUCTION}\label{section:intro}

We assume $\mathsf{ZF} + \mathsf{DC}$
as our background theory unless otherwise stated. (However, we will sometimes weaken our choice principle to a fragment of $\mathsf{DC}$.)
In this setting, it is possible for $\omega_1$ to exhibit ``large cardinal'' properties such as strong compactness.
The appropriate definition of strong compactness is made in terms of measures (ultrafilters) on sets of the form $\powerset_{\omega_1}(X)$.

\begin{definition}
 Let $X$ be an uncountable set.
 A measure $\mu$ on $\powerset_{\omega_1}(X)$ is \emph{countably complete} if it is closed under countable intersections and \emph{fine} if it contains the set $\{\sigma \in \powerset_{\omega_1}(X) : x \in \sigma\}$ for all $x \in X$.
 We say that $\omega_1$ is \emph{$X$-strongly compact} if there is a countably complete fine measure on $\powerset_{\omega_1}(X)$.
\end{definition}

For uncountable sets $X$ and $Y$,
we will often use the elementary fact that if
$\omega_1$ is $X$-strongly compact and there is a surjection from $X$ to $Y$, then $\omega_1$ is $Y$-strongly compact.

In the absence of $\mathsf{AC}$, it may become necessary to consider degrees $X$ of strong compactness that are not wellordered.
The first and most important example is $X=\mathbb{R}$.
The theory $\mathsf{ZFC} + {}$``there is a measurable cardinal'' is equiconsistent with the theory  $\mathsf{ZF} + \mathsf{DC} + {}$``$\omega_1$ is $\mathbb{R}$-strongly compact.'' (For a proof of the forward direction, see Trang \cite{trang2012structureLRmu}.
The reverse direction is proved by noting that $\omega_1$ is $\omega_1$-strongly compact, hence measurable, and considering an inner model $L(\mu)$ where $\mu$ is a measure on $\omega_1$.)

Another way to obtain $\mathbb{R}$-strong compactness of $\omega_1$
that is more relevant to this paper is by the Axiom of Determinacy.
If $\mathsf{AD}$ holds then by Martin's cone theorem, for every set $A \in \powerset_{\omega_1}(\mathbb{R})$ the property $\{x \in \mathbb{R} : x \le_\text{T} d\} \in A$ either holds for a cone of Turing degrees $d$ or fails for a cone of Turing degrees $d$, giving a countably complete fine measure on $\powerset_{\omega_1}(\mathbb{R})$.

Besides $\mathbb{R}$, another relevant degree of strong compactness is the
cardinal $\Theta$, which is defined as the least ordinal that is not a surjective image of $\mathbb{R}$.
In other words, $\Theta$ is the successor of the continuum in the sense of surjections.
If the continuum can be wellordered then this is the same as the successor in the sense of injections (that is, $\mathfrak{c}^+$.)
However in general it can be much larger.
For example, if $\mathsf{AD}$ holds then $\Theta$ is strongly inaccessible by Moschovakis's coding lemma, but on the other hand there is no injection from $\omega_1$ into $\mathbb{R}$.

If $\omega_1$ is $\mathbb{R}$-strongly compact, then pushing forward a measure witnessing this by surjections, we see that
$\omega_1$ is $\lambda$-strongly compact for every uncountable cardinal $\lambda < \Theta$.
In general all we can say is $\Theta \ge \omega_2$ and so this does not given anything beyond measurability of $\omega_1$.
However, it does suggest two marginal strengthings of our hypothesis on $\omega_1$ with the potential to increase the consitency strength beyond measurability.  Namely, we may add the hypothesis ``$\omega_1$ is $\omega_2$-strongly compact'' or the hypothesis ``$\omega_1$ is $\Theta$-strongly compact.''  We consider both strengthenings and obtain equiconsistency results in both cases.

In order to state and obtain sharper results, we first recall some combinatorial consequences of strong compactness.
Let $\lambda$ be an infinite cardinal and let $\vec{C} = (C_\alpha : \alpha \in \lim(\lambda))$ be a sequence such that each set $C_\alpha$ is a club subset of $\alpha$.
The sequence $\vec{C}$ is \emph{coherent} if for all $\beta \in \lim(\lambda)$ and all $\alpha \in \lim(C_\beta)$ we have $C_\alpha = C_\beta \cap \alpha$.
A \emph{thread} for a coherent sequence $\vec{C}$ is a club subset $D \cup \lambda$ such that for all $\alpha \in \lim(D)$ we have $C_\alpha = D \cap \alpha$.
An infinite cardinal $\lambda$ is called \emph{threadable} if every coherent sequence of length $\lambda$ has a thread.\footnote{Threadability of $\lambda$ is also known as $\neg \square(\lambda)$.}

The following result is a well-known consequence of the ``discontinuous ultrapower'' characterization of strong compactness. However, without $\mathsf{AC}$ \Los's theorem may fail for ultrapowers of $V$, so we must verify that the argument can be done using ultrapowers of appropriate inner models instead.

\begin{lemma}\label{lemma:threadability-from-strong-compactness}
 Assume $\mathsf{ZF} + \mathsf{DC} + {}$``$\omega_1$ is $\lambda$-strongly compact'' where
 $\lambda$ is a cardinal of uncountable cofinality.  Then $\lambda$ is threadable.
\end{lemma}
\begin{proof}
  Let $\vec{C} = (C_\alpha : \alpha \in \lim(\lambda))$ be a coherent sequence
 such that each set $C_\alpha$ is a club in $\alpha$.
 Let $\mu$ be a countably complete fine measure on $\powerset_{\omega_1}(\lambda)$ and let $j : V \to \Ult(V,\mu)$ be the ultrapower map corresponding to $\mu$.
 The restriction $j \restriction L[\vec{C}]$ is an elementary embedding from $L[\vec{C}]$ to $L[j(\vec{C})]$.
 Note that we must use all functions from $V$ in our ultrapower rather than only using functions from $L[\vec{C}]$
 because the measure $\mu$ might not concentrate on $\powerset_{\omega_1}(\lambda) \cap L[\vec{C}]$.
 Note also that $j$ is discontinuous at $\lambda$:
 for any ordinal $\alpha < \lambda$,
 we have $j(\alpha) < [\sigma \mapsto \sup \sigma]_\mu < j(\lambda)$ where the first inequality holds because $\mu$ is fine and the second inequality holds because $\lambda$ has uncountable cofinality.  
  
 Now the argument continues as usual.
 We define the ordinal $\gamma = \sup j[\lambda]$ and note that $\gamma < j(\lambda)$ and that $j[\lambda]$
 is an $\omega$-club in $\lambda$.\footnote{We seem to need $\mathsf{DC}$ in this argument to see that the ultrapower is wellfounded and in particular that $\sup j[\lambda]$ exists.}
 Therefore the set $j[\lambda] \cap \lim(j(\vec{C})_\gamma)$ is unbounded in $\gamma$,
 so its preimage $S = j^{-1}[\lim(j(\vec{C})_\gamma)]$ is unbounded in $\lambda$.
 Note that the club $C_\alpha$ is an initial segment of $C_\beta$ whenever $\alpha,\beta \in S$ and $\alpha < \beta$;
 this is easy to check using the elementarity of $j \restriction L[\vec{C}]$ and the coherence of $j(\vec{C})$.
 Therefore the union of clubs $\bigcup_{\alpha \in S} C_\alpha$ threads the sequence $\vec{C}$.
\end{proof}

The following lemma is almost an immediate consequence except that we want to weaken the hypothesis $\mathsf{DC}$ a bit.

\begin{lemma}\label{lemma:threadability-from-R-strong-compactness}
 Assume $\mathsf{ZF} + \mathsf{DC}_\mathbb{R} + {}$``$\omega_1$ is $\mathbb{R}$-strongly compact.''
 Let $\lambda < \Theta$ be a cardinal of uncountable cofinality.  Then $\lambda$ is threadable.
\end{lemma}
\begin{proof}
 Let $\vec{C}$ be a coherent sequence of length $\lambda$.
 First, note that we may pass to an inner model containing $\vec{C}$ where $\mathsf{DC}$ holds in addition to our other hypotheses.
 Namely, let $f: \mathbb{R} \to \lambda$ be a surjection,
 let $\mu$ be a fine, countably complete measure on $\powerset_{\omega_1}(\mathbb{R})$,
 let $C = \{(\alpha,\beta) : \alpha \in C_\beta\}$, and consider the model
 $M = L(\mathbb{R})[f,\mu,C]$, where the square brackets indicate that we are constructing from $f$, $\mu$ and $C$ as predicates.
 (In the case of $\mu$, this distinction is important: we are not putting all elements of $\mu$ into the model.)

 It can be easily verified that all of our hypotheses are downward absolute to the model $M$, and that our desired conclusion that $\vec{C}$ has a thread is upward absolute from $M$ to $V$.  In the model $M$ every set is a surjective image of $\mathbb{R}\times \alpha$ for some ordinal $\alpha$, so $\mathsf{DC}$ follows from $\mathsf{DC}_\mathbb{R}$ by a standard argument.
 Moreover, $\omega_1$ is $\lambda$-strongly compact in $M$ by pushing forward the measure $\mu$ (restricted to $M$) by the surjection $f$, so the desired result follows from Lemma \ref{lemma:threadability-from-strong-compactness}.
\end{proof}

A further combinatorial consequence of strong compactness of $\omega_1$ is the failure of Jensen's square principle $\square_{\omega_1}$.  In fact
$\neg \square_{\omega_1}$ follows from the assumption that $\omega_2$ is threadable or singular (note that successor cardinals may be singular in the absence of $\mathsf{AC}$.)

\begin{lemma}\label{lemma:failure-of-square-from-threadability}
 Assume $\mathsf{ZF}$.
 If $\omega_2$ is singular or threadable, then $\neg \square_{\omega_1}$.
\end{lemma}
\begin{proof}
 Suppose toward a contradiction that $\omega_2$ is singular or threadable and we have a $\square_{\omega_1}$-sequence
 $(C_\alpha : \alpha \in \lim(\omega_2))$.
 If $\omega_2$ is singular, we do not need coherence of the sequence to reach a contradiction. Take any cofinal set $C_{\omega_2}$ in $\omega_2$ of order type $\le \omega_1$ and
 recursively define a sequence of functions $(f_\alpha : \alpha \in [\omega_1,\omega_2])$ such that each function $f_\alpha$ is a surjection from $\omega_1$ onto $\alpha$, using our small cofinal sets $C_\alpha$ at limit stages.
 Then the function $f_{\omega_2}$ is a surjection from $\omega_1$ onto $\omega_2$, a contradiction.
 On the other hand, if $\omega_2$ is regular and threadable, take a thread $C_{\omega_2}$ through the square sequence.  Then by the usual argument the order type of $C_{\omega_2}$ is at most $\omega_1 + \omega$, contradicting the regularity of $\omega_2$.
\end{proof}

Now we can state our equiconsistency results and prove their easier directions.

\begin{theorem}\label{theorem:AD-equiconsistency}
The following theories are equiconsistent:
\begin{enumerate}
\item\label{item-AD} $\mathsf{ZF} + \mathsf{DC} + \mathsf{AD}$.


\item\label{item-P-omega-1-strongly-compact} $\mathsf{ZF} + \mathsf{DC} + {}$``$\omega_1$ is $\powerset(\omega_1)$-strongly compact.''

\item\label{item-omega-2-strongly-compact} $\mathsf{ZF} + \mathsf{DC} + {}$``$\omega_1$ is $\mathbb{R}$-strongly compact and $\omega_2$-strongly compact.''

\item\label{item-not-square-omega-1}
$\mathsf{ZF} + \mathsf{DC} + {}$``$\omega_1$ is $\mathbb{R}$-strongly compact and $\neg\square_{\omega_1}$.''

\end{enumerate}
\end{theorem}
\begin{proof}
 \eqref{item-AD} $\implies$ \eqref{item-P-omega-1-strongly-compact}:
 Under $\mathsf{AD}$, Martin's cone theorem implies that $\omega_1$ is $\mathbb{R}$-strongly compact.
 There is a surjection from $\mathbb{R}$ onto $\powerset(\omega_1)$ by 
 Moschovakis's coding lemma, so $\omega_1$ is $\powerset(\omega_1)$-strongly compact as well.

 \eqref{item-P-omega-1-strongly-compact} $\implies$ \eqref{item-omega-2-strongly-compact}:
 This follows from the existence of surjections from $\powerset(\omega_1)$ onto $\mathbb{R}$ and $\omega_2$.
 
 \eqref{item-omega-2-strongly-compact} $\implies$ \eqref{item-not-square-omega-1}:
 This follows from Lemmas \ref{lemma:threadability-from-strong-compactness} and \ref{lemma:failure-of-square-from-threadability}.
 
 $\Con \eqref{item-not-square-omega-1} \implies \Con \eqref{item-AD}$:
 In a later section, we will show that $\eqref{item-not-square-omega-1} \implies \mathsf{AD}^{L(\mathbb{R})}$.
\end{proof}

Moving up the consistency strength hierarchy, the next natural target for equiconsistency is the theory
$\sf{ZF} + \sf{AD}_\mathbb{R}$. Here $\sf{AD}_\mathbb{R}$ denotes the Axiom of Determinacy for real games, which has higher consistency strength than $\mathsf{AD}$ and cannot hold in $L(\mathbb{R})$.
To get a model of $\mathsf{AD}_\mathbb{R}$ we will need to augment our hypothesis somehow.

The consistency strength of $\mathsf{AD}_\mathbb{R}$, the Axiom of Real Determinacy, is sensitive to $\mathsf{DC}$, so for our next result we must weaken $\mathsf{DC}$ somewhat.
(By contrast, the theory $\mathsf{ZF} + \mathsf{DC} + \mathsf{AD}$ is equiconsistent with $\mathsf{ZF} + \mathsf{AD}$ by a theorem of Kechris.)  By $\mathsf{DC}_{\powerset({\omega_1})}$ we will denote the fragment of $\mathsf{DC}$ that allows us to choose $\omega$-sequences of subsets of $\omega_1$.

\begin{theorem}\label{theorem:ADR-equiconsistency}
The following theories are equiconsistent:
\begin{enumerate}
\item\label{item-ADR}
$\mathsf{ZF} + \mathsf{AD}_\mathbb{R}$.

\item\label{item-R-strongly-compact-and-theta-singular}
$\mathsf{ZF} + \mathsf{DC}_{\powerset(\omega_1)} + {}$``$\omega_1$ is $\mathbb{R}$-strongly compact and $\Theta$ is singular.''
\end{enumerate}
\end{theorem}
\begin{proof}
 $\Con\eqref{item-ADR} \implies \Con\eqref{item-R-strongly-compact-and-theta-singular}$:
 By Solovay \cite{solovay1978independence},
 if $\mathsf{ZF} + \mathsf{AD}_\mathbb{R}$ is consistent then so is
 $\mathsf{ZF} + \mathsf{AD}_\mathbb{R} + {}$``$\Theta$ is singular.'' (In particular Solovay showed that the cofinality of $\Theta$ can be countable, which implies the failure of $\mathsf{DC}$.)
 Under $\mathsf{AD}_\mathbb{R}$  we have that $\omega_1$ is $\mathbb{R}$-strongly compact by Martin's measure
 (this just follows from $\mathsf{AD}$)
 and we have $\mathsf{DC}_\mathbb{R}$ (this follows from uniformization for total relations on $\mathbb{R}$.)
 Moreover there is a surjection from $\mathbb{R}$ to $\powerset(\omega_1)$ by the coding lemma, so $\mathsf{DC}_{\mathbb{R}}$ can be strengthened to $\mathsf{DC}_{\powerset(\omega_1)}$.

 $\Con\eqref{item-R-strongly-compact-and-theta-singular} \implies \Con \eqref{item-ADR}$:
 In a later section, we will show that if statement \eqref{item-R-strongly-compact-and-theta-singular} holds, then
 statement \eqref{item-ADR} holds in an inner model of the form $L(\Omega^*, \mathbb{R})$ where $\Omega^* \subset \powerset(\mathbb{R})$.
 Note that statement \eqref{item-R-strongly-compact-and-theta-singular} implies that $\omega_2$ is either singular (if $\omega_2 = \Theta$)
 or threadable (if $\omega_2 < \Theta$, by Lemma \ref{lemma:threadability-from-R-strong-compactness}) so in either case we have 
 $\neg \square_{\omega_1}$ by Lemma \ref{lemma:failure-of-square-from-threadability}.
 Therefore we can make some use of the argument for $\Con \eqref{item-not-square-omega-1} \implies \Con \eqref{item-AD} $ of Theorem \ref{theorem:AD-equiconsistency} here, once we check that $\mathsf{DC}_{\powerset(\omega_1)}$ suffices in place of $\mathsf{DC}$ for this argument.
\end{proof}

Adding back full $\mathsf{DC}$, we will obtain an equiconsistency result at a higher level.

\begin{theorem}\label{theorem:ADR-DC-equiconsistency}
 The following theories are equiconsistent:
 \begin{enumerate}
  \item\label{item-ADR-DC}
  $\mathsf{ZF} + \mathsf{DC} + \mathsf{AD}_\mathbb{R} $.

  \item\label{item-PR-strongly-compact-and-DC}
  $\mathsf{ZF} + \mathsf{DC} + {}$``$\omega_1$ is $\powerset(\mathbb{R})$-strongly compact.''

  \item\label{item-R-strongly-compact-and-Theta-strongly-compact}
  $\mathsf{ZF} + \mathsf{DC} + {}$``$\omega_1$ is $\mathbb{R}$-strongly compact and $\Theta$-strongly compact.''

  \item\label{item-R-strongly-compact-and-theta-singular-and-DC}
  $\mathsf{ZF} + \mathsf{DC} + {}$``$\omega_1$ is $\mathbb{R}$-strongly compact and $\Theta$ is singular.''
 \end{enumerate}
\end{theorem}
\begin{proof}
 $\Con\eqref{item-ADR-DC} \implies \Con \eqref{item-PR-strongly-compact-and-DC}$:
 By Solovay \cite{solovay1978independence},
 under $\mathsf{ZF} + \mathsf{AD}_\mathbb{R}$ we have $\mathsf{DC}$ if and only if $\Theta$ has uncountable cofinality, and in a minimal model of $\mathsf{ZF} + \mathsf{DC} + \mathsf{AD}_\mathbb{R}$ we have that $\Theta$ is singular of cofinality $\omega_1$.
 Assume that we are in such a minimal model of $\mathsf{ZF} + \mathsf{DC} + \mathsf{AD}_\mathbb{R}$ and
 take a cofinal increasing function
 $\pi:\omega_1 \rightarrow \Theta$.

 We can express $\powerset(\mathbb{R})$ as an increasing union
 $\bigcup_{\alpha < \omega_1} \Gamma_\alpha$ where the pointclass $\Gamma_\alpha$ consists of all sets of reals of Wadge rank at most $\pi(\alpha)$.
 For each $\alpha < \omega_1$ there is a surjection from $\mathbb{R}$ onto $\Gamma_\alpha$, so $\omega_1$ is $\Gamma_\alpha$-strongly compact.
 Moreover, $\mathsf{AD}_\mathbb{R}$ implies that there is a uniform way to choose, for each $\alpha < \omega_1$, a countably complete fine measure $\mu_\alpha$ on $\powerset_{\omega_1}(\Gamma_\alpha)$ witnessing this fact (namely the unique normal fine measure; see Woodin \cite[Theorem 4]{woodin1983ad}.)

 Using a countably complete nonprincipal measure $\nu$ on $\omega_1$ (which exists because $\omega_1$ is $\omega_1$-strongly compact) we can assemble these measures $\mu_\alpha$ into a countably complete fine measure $\mu^*$ on $\powerset_{\omega_1}(\powerset(\mathbb{R}))$ as follows: for $A\subseteq \powerset_{\omega_1}(\powerset(\mathbb{R}))$, we say
 \begin{center}
  $A\in \mu^* \iff \forall^*_\nu \alpha\, A \cap \Gamma_\alpha \in \mu_\alpha$.
 \end{center}
 It's easy to verify that $\mu^*$ is countably complete and fine using the fact that the measure $\nu$ and the measures $\mu_\alpha$ are countably complete and nonprincipal/fine respectively. Therefore the measure $\mu^*$ witnesses that $\omega_1$ is $\powerset(\mathbb{R})$-strongly compact, so statement \eqref{item-PR-strongly-compact-and-DC} holds (in our minimal model of $\mathsf{ZF} + \mathsf{DC} + \mathsf{AD}_\mathbb{R}$.)

 \eqref{item-PR-strongly-compact-and-DC} $\implies$ \eqref{item-R-strongly-compact-and-Theta-strongly-compact}: This follows from the existence of surjections from $\powerset(\mathbb{R})$ onto $\mathbb{R}$ and $\Theta$.

 $\Con\eqref{item-ADR-DC} \implies \Con \eqref{item-R-strongly-compact-and-theta-singular-and-DC}$:
 This follows by the aforementioned result of Solovay that, in a minimal model of $\mathsf{ZF} + \mathsf{DC} + \mathsf{AD}_\mathbb{R}$ the cardinal $\Theta$ is singular of cofinality $\omega_1$ (and of course $\omega_1$ is $\mathbb{R}$-strongly compact by Martin's measure.)

 $\Con(\ref{item-R-strongly-compact-and-Theta-strongly-compact}) \vee \Con(\ref{item-R-strongly-compact-and-theta-singular-and-DC}) \implies \Con(\ref{item-ADR-DC})$:
 We
 will show in a later section that
 if either statement \eqref{item-R-strongly-compact-and-Theta-strongly-compact} or statement \eqref{item-R-strongly-compact-and-theta-singular-and-DC} holds, then
 statement \eqref{item-ADR} holds in an inner model of the form $L(\Omega^*, \mathbb{R})$ where $\Omega^* \subset \powerset(\mathbb{R})$.
 The proof of 
$\Con(\ref{item-R-strongly-compact-and-theta-singular-and-DC}) \implies \Con(\ref{item-ADR-DC})$
 is similar to the proof of $\Con\eqref{item-R-strongly-compact-and-theta-singular} \implies \Con \eqref{item-ADR}$ in Theorem \ref{theorem:ADR-DC-equiconsistency}, although one should note that the inner model $L(\Omega^*, \mathbb{R})$ does not simply absorb $\mathsf{DC}$ from $V$; a bit more argument is required.
\end{proof}

\section{FRAMEWORK FOR THE CORE MODEL INDUCTION}
This section is an adaptation of the framework for the core model induction developed in \cite{trang2012scales} and \cite{operator_mice}, which in turns build on earlier formulations in \cite{CMI}. For more detailed discussions on the notions defined below as well as results concerning them, see \cite{trang2012scales} and \cite{operator_mice}. The first subsection imports some terminology from the theory of hybrid mice developed in \cite{trang2012scales} and \cite{operator_mice}. The terminology in this subsection will be used in Subsection \ref{cmi_operators} to define core model induction operators and  will be needed in many other places in the paper. The reader may skip them on the first read and come back when needed. Subsection \ref{briefHodMice} summarizes the theory of hod mice developed in \cite{ATHM}. Subsection \ref{cmi_operators} defines core model induction operators which are the operators we will construct in this paper.

\subsection{$\Omega$-PREMICE, STRATEGY PREMICE, AND G-ORGANIZED $\Omega$-PREMICE}\label{g_organized}
For a complete theory of $\F$-premice for operators $\F$, the reader is advised to read \cite{operator_mice}; for a detailed treatment of strategy mice, the reader is advised to read \cite[Sections 2,3]{trang2012scales}. We will use the terminology from these sources from now on.\footnote{The theory of strategy mice can be developed as a special case of the general theory of operator mice in \cite{operator_mice} but the authors of the papers decided to define strategy mice as $\J$-structures as this approach seems more convenient and gave the right notation for proving strong condensation properties of strategy mice like \cite[Lemma 4.1]{trang2012scales}.}

The definition below is essentially \cite[Definition 3.8]{trang2012scales}. For explanations about the notations, see \cite[Sections 2,3]{trang2012scales}.

\begin{definition}\label{dfn:determines}
Let $t=(\Omega,\varphi,X,A,\kappa)$ be suitable (see \cite[Definition 3.4]{trang2012scales} and $\mathfrak{M}=\M_1^{X,\#}(A)$.
We say that $\mathfrak{M}$ $\rm{generically \ interprets}$ $\Omega$\footnote{In \cite[Definition 3.8]{trang2012scales}, the terminology is: $t$ determines itself on generic extensions. We will later define a notion of generic determination which is slightly different.} iff there
are formulas $\Phi,\Psi$ in
$\Ll^+$ and some $\gamma > \delta^\mathfrak{M}$ such that 
$\mathfrak{M}|\gamma
\vDash\Phi$ and for any non-dropping
$\Lambda^{X,\kappa}_\mathfrak{M}$-iterate $\N$ of
$\mathfrak{M}$ via a countable tree $\T$ based on $\mathfrak{M}|\delta^{\mathfrak{M}}$,\footnote{$\delta^{\mathfrak{M}}$ is the Woodin cardinal of $\mathfrak{M}$ and $\Lambda^{X,\kappa}_\MFsharp$ denotes the unique $X$-$(0,\kappa)$-iteration strategy for $\MFsharp$.} any $\N$-cardinal 
$\delta$, 
any
$\gamma\in\Ord$ such that $\N|\gamma\models\Phi\ \&\ $``$\delta$ is
Woodin'', and any $g$ which is set-generic over $\N|\gamma$ (with $g\in V$),
we have that $\R =_{\rm{def}}$$ (\N|\gamma)[g]$
is closed under $\Omega$, and $\Omega\rest\R$ is defined 
over
$\R$ by
$\Psi$. We say such a pair $(\Phi,\Psi)$ $\rm{generically \ determines }$
t (or just $\Omega$).

Let $A\in\HC$ and let $\Omega$ be either an operator or an iteration strategy.
We say that $(\Omega,A)$(or just $\Omega$) is $\rm{nice}$ iff $(\Omega,A)$ is 
suitable and $(t_{\Omega,A})_2$ generically interprets $\Omega$.\footnote{$t_{\Omega,A}$ is a $5$-tuple defined \cite[page 27]{trang2012scales} and $(t_{\Omega,A})_2$ is the third component of $t_{\Omega,A}$.} We say that $(\Phi,\Psi)$ 
$\rm{generically\ determines}$
$(\Omega,A)$ iff $(\Phi,\Psi)$ generically determines $t_{\Omega,A}$.
\end{definition}

We fix a nice $(\Omega,A)$ (or just nice $\Omega$; we will at times ignore $A$), $X = (t_{\Omega,A})_2$, $\MFsharp$, $\Lambda_{\MFsharp} = \Lambda$, $(\Phi,\Psi)$ for the rest of the section. We define $\M_1^X(A)$ from $\MFsharp$ in the standard way.

See \cite[Section 3]{trang2012scales} for a proof that if $\Omega=\Sigma$ is a strategy (of a hod mouse, a suitable mouse) with branch condensation and is fullness preserving with respect to mice in some sufficiently closed, determined pointclass $\Gamma$ or if $\Sigma$ is the unique strategy of a sound ($Y$)-mouse for some operator $Y$, $\M_1^{Y,\sharp}$ generically interprets $Y$, and $Y$ condenses finely (see \cite[Definition 3.18]{operator_mice}) then $\MFsharp$ generically interprets $\Omega$.

\begin{definition}[Sargsyan, \cite{ATHM}]\label{genGenTree}
Let $M$ be a transitive structure. Let $\dot{G}$ be the name for the generic 
$G\subseteq\Coll(\omega,M)$ and let $\dot{x}_{\dot{G}}$ be the
canonical name for the real coding $\{(n,m) \ | \ G(n) \in G(m)\}$, where we
identify $G$ with $\bigcup G$. The 
\textbf{tree
$\Tt_M$ for making $M$
generically generic}, is the iteration tree $\Tt$ on $\MFsharp$ of maximal 
length such
that:
\begin{enumerate}
\item $\Tt$ is via $\Lambda$ and is everywhere non-dropping.
\item $\Tt\rest\OR(M)+1$ is the tree given by linearly iterating the first total 
measure of $\MFsharp$ and its images.
\item Suppose $\lh(\Tt)\geq\OR(M)+2$ and let $\alpha+1\in(\OR(M),\lh(\Tt))$.
Let $\delta=\delta(\M^\Tt_\alpha)$ and let $\BB=\BB(M^\Tt_\alpha)$ be the 
extender algebra of
$M^\Tt_\alpha$ at $\delta$. Then $E^\mathcal{T}_\alpha$ is the extender $E$ with
least index in $M^\mathcal{T}_\alpha$ such that for some condition
$p\in\Coll(\omega,M)$, $p \Vdash$``There is a $\mathbb{B}$-axiom induced
by $E$ which fails for $\dot{x}_{\dot{G}}$''.
\end{enumerate}
Assuming that $\MFsharp$ is sufficiently iterable, then $\Tt_M$ exists and has 
successor length.
\end{definition}

Sargsyan noticed that one can feed in 
$\Fop$ into a structure $\N$ indirectly, by feeding in the branches for 
(initial segments of) $\Tt_\M$, for various $\M\ins\N$. The operator ${^\g\Omega}$, defined in \cite[Definition 3.42]{trang2012scales}, and 
used in building g-organized $\Omega$-premice, feeds in 
branches for such $\Tt_\M$'s. We will also ensure that being such a structure is 
first-order - other than wellfoundedness and the correctness of the branches - 
by allowing sufficient spacing between these branches (see \cite[Remark 3.37]{trang2012scales}). 

\cite{trang2012scales} then defines the notions of g-organized $\Omega$-premouse and $\Theta$-g-organized $\Omega$-premouse. The reader can again see \cite[Section 3]{trang2012scales} for a more extensive treatment of these notions.

$\Theta$-g-organized $\Omega$-mice over $\mathbb{R}$ are important in the scales analysis generalizing Steel's work in Lp$(\mathbb{R})$. 

\begin{definition}\label{dfn:self-scaled}
Let $Y\sub\mathbb{R}$. We say that 
$Y$ is $\rm{self-scaled}$ iff there are scales on $Y$ and $\RR\cut 
Y$ which are analytical (i.e., $\Sigma^1_n$ for some $n<\om$) in $Y$.
\end{definition}

\begin{definition}\label{Lp}
Suppose $\Omega$ is a nice (operator or iteration strategy) and is an iteration strategy and $Y\subseteq\mathbb{R}$ is self-scaled. We define 
$\Lp^{^\gTheta\Omega}(\mathbb{R},Y)$ as the stack of all $\Theta$-g-organized $\Omega$-mice $\N$ 
over  
$(H_{\om_1},Y)$ (with parameter $\MFsharp$). We also say ($\Theta$-g-organized) $\Omega$-premouse over 
($\mathbb{R}$,Y) to in fact
mean over ($H_{\omega_1}, Y$).\end{definition}
\begin{remark}\label{SamePR}
It's not hard to see that for any such $Y$ as in Definition \ref{Lp}, $\powerset(\mathbb{R})\cap \textrm{Lp}^{^\g\Omega}(\mathbb{R},Y) = \powerset(\mathbb{R})\cap \textrm{Lp}^{^\gTheta\Omega}(\mathbb{R},Y)$. Suppose $\M$ is an initial segment of the first hierarchy and $\M$ is $E$-active. Note that $\M\vDash ``\Theta$ exists" and $\M|\Theta$ is $\Omega$-closed. By induction below $\M|\Theta^\M$, $\M|\Theta^\M$ can be rearranged into an initial segment $\N'$ of the second hierarchy. Above $\Theta^\M$, we simply copy the $E$-sequence and $B$-sequence\footnote{The $E$-sequence is the extender sequence of $\M$ and the $B$-sequence codes fragments of the strategy of $\MFsharp$.} from $\M$ over to obtain an $\N\lhd \textrm{Lp}^{^\gTheta\Fop}(\mathbb{R},X)$ extending $\N'$. The converse is similar. Similarly, if $\Omega$ is such that Lp$^{\Omega}(\mathbb{R},Y)$ is well-defined and $\Omega$ relativizes well, then $\powerset(\mathbb{R})\cap \textrm{Lp}^{^\g\Omega}(\mathbb{R},Y)=\powerset(\mathbb{R})\cap \textrm{Lp}^{\Omega}(\mathbb{R},Y)$. See \cite[Remark 4.11]{trang2012scales}. 
\end{remark}
In core model induction applications, we often have a pair $(\P,\Sigma)$
where $\P$ is a hod premouse and $\Sigma$ is $\P$'s strategy with branch
condensation and is fullness preserving (relative to mice in some pointclass) or
$\P$ is a sound (hybrid) premouse projecting to some countable set $a$ and
$\Sigma$ is the unique (normal) $\omega_1+1$-strategy for $\P$.  Let $\Omega = \Sigma$, $A\in \HC$ transitive such that $\P\in\J_1(A)$, $X$ be defined from $(\Omega,A)$ as above, and suppose $\MFsharp = \M_1^{X,\sharp}(A)$ exists. \cite{trang2012scales} shows that $\Omega$
condenses finely and $\MFsharp$ generically interprets $(\Omega,A)$. Also, the core model induction will give us that the code of $\Omega$, Code$(\Omega)$ (under a natural coding of subsets of $\HC$ by subsets of $\mathbb{R}$) is self-scaled. Thus, we can 
define $\Lp^{^\gTheta\Omega}(\mathbb{R},\rm{Code}$$(\Omega))$ as above (assuming sufficient iterability of $\MFsharp$). A core model induction 
is then used to prove that there is a maximal constructibly closed initial segment $\M$ of
$\Lp^{^\gTheta\Omega}(\mathbb{R},\rm{Code}$$(\Omega))$ that satisfies $\textsf{AD}^+$. What's needed
to prove this is the scales analysis of $\Lp^{^\gTheta\Omega}(\mathbb{R},\rm{Code}$$(\Omega))$
from the optimal hypothesis (similar to those used by Steel; see
\cite{K(R)} and \cite{Scalesendgap}). This is carried out in \cite{trang2012scales}; we will not go into details here.

\subsection{A VERY BRIEF TALE OF HOD MICE}
\label{briefHodMice} 
In this paper, a hod premouse $\P$ is one defined as in \cite{ATHM}. The reader is advised to consult \cite{ATHM} for basic results and notations concerning hod premice and mice. Let us mention some basic first-order properties of a hod premouse $\P$. There are an ordinal $\lambda^\P$ and sequences $\langle(\P(\alpha),\Sigma^\P_\alpha) \ | \ \alpha < \lambda^\P\rangle$ and $\langle \delta^\P_\alpha \ | \ \alpha \leq \lambda^\P  \rangle$ such that 
\begin{enumerate}
\item $\langle \delta^\P_\alpha \ | \ \alpha \leq \lambda^\P  \rangle$ is increasing and continuous and if $\alpha$ is a successor ordinal then $\P \vDash \delta^\P_\alpha$ is Woodin;
\item $\P(0) = \rm{Lp}_\omega(\P|\delta_0)^\P$; for $\alpha < \lambda^\P$, $\P(\alpha+1) = (\rm{Lp}$$_\omega^{^g\Sigma^\P_\alpha}(\P|\delta_\alpha))^\P$; for limit $\alpha\leq \lambda^\P$, $\P(\alpha) = (\rm{Lp}$$_\omega^{^g\oplus_{\beta<\alpha}\Sigma^\P_\beta}(\P|\delta_\alpha))^\P$;
\item $\P \vDash \Sigma^\P_\alpha$ is a $(\omega,o(\P),o(\P))$\footnote{This just means $\Sigma^\P_\alpha$ acts on all stacks of $\omega$-maximal, normal trees in $\P$.}-strategy for $\P(\alpha)$ with hull condensation;
\item if $\alpha < \beta < \lambda^\P$ then $\Sigma^\P_\beta$ extends $\Sigma^\P_\alpha$.
\end{enumerate}

Hod mice in this paper are g-organized; this is so that $S$-constructions work out smoothly as in the pure $L[\es]$-case. We will write $\delta^\P$ for $\delta^\P_{\lambda^\P}$ and $\Sigma^\P=\oplus_{\beta<\lambda^\P}\Sigma^\P_{\beta}$. Note that $\P(0)$ is a pure extender model. Suppose $\P$ and $\Q$ are two hod premice. Then $\P\trianglelefteq_{hod}\Q$\index{$\trianglelefteq_{hod}$} if there is $\a\leq\l^\Q$ such that $\P=\Q(\a)$. We say then that $\P$ is a \textit{hod initial segment} of $\Q$. $(\P,\Sigma)$ is a \textit{hod pair} if $\P$ is a hod premouse and $\Sigma$ is a strategy for $\P$ (acting on countable stacks of countable normal trees) such that $\Sigma^\P \subseteq \Sigma$ and this fact is preserved under $\Sigma$-iterations. Typically, we will construct hod pairs $(\P,\Sigma)$ such that $\Sigma$ has hull condensation, branch condensation, and is $\Gamma$-fullness preserving for some pointclass $\Gamma$. 

The reader should consult \cite{ATHM} for the definition of $B(\Q,\Sigma)$, and $I(\Q,\Sigma)$. Roughly speaking, $B(\Q,\Sigma)$ is the collection of all hod pairs which are strict hod initial segments of a $\Sigma$-iterate of $\Q$ and $I(\Q,\Sigma)$ is the collection of all $\Sigma$-iterates of $\Sigma$. In the case $\lambda^\Q$ is limit, $\Gamma(\Q,\Sigma)$ is the collection of $A\subseteq \mathbb{R}$ such that $A$ is Wadge reducible to some $\Psi$ for which there is some $\R$ such that $(\R,\Psi)\in B(\Q,\Sigma)$. See \cite{ATHM} for the definition of $\Gamma(\Q,\Sigma)$ in the case $\lambda^\Q$ is a successor ordinal.

\cite{ATHM} constructs under $\textsf{AD}^+$ and the hypothesis that there are no models of ``$\textsf{AD}_\mathbb{R}+\Theta$ is regular"  hod pairs that are fullness preserving, positional, commuting, and have branch condensation. Such hod pairs are particularly important for our computation as they are points in the direct limit system giving rise to \textrm{HOD} of $\textsf{AD}^+$ models. Under $\sf{AD}^+$, for hod pairs $(\M_\Sigma, \Sigma)$, if $\Sigma$ is a strategy with branch condensation and $\VT$ is a stack on $\M_\Sigma$ with last model $\N$, $\Sigma_{\N, \VT}$ is independent of $\VT$. Therefore, later on we will omit the subscript $\VT$ from $\Sigma_{N, \VT}$ whenever $\Sigma$ is a strategy with branch condensation and $\M_\Sigma$ is a hod mouse. In a core model induction, we don't quite have, at the moment $(\M_\Sigma,\Sigma)$ is constructed, an $\textsf{AD}^+$-model $M$ such that $(\M_\Sigma,\Sigma)\in M$ but we do know that every $(\R,\Lambda)\in B(\M_\Sigma,\Sigma)$ belongs to such a model. We then can show (using our hypothesis) that $(\M_\Sigma,\Sigma)$ belongs to an $\textsf{AD}^+$-model.

\subsection{CORE MODEL INDUCTION OPERATORS}
\label{cmi_operators}
Let 
\begin{center}
$\Omega^* =  \{A \subseteq \mathbb{R} \ | \ L(A,\mathbb{R})\vDash \sf{AD}^+ \}$.
\end{center}

We assume, for contradiction that 
\begin{adjustwidth}{2cm}{2cm}

$(\dag): \ \ \ $ there is no model $M$ containing all reals and ordinals such that $M \vDash ``\sf{AD}$$_\mathbb{R} + \Theta$ is regular".

\end{adjustwidth}
Under this smallness assumption, by work of G. Sargsyan in \cite{ATHM}, $\Omega^*$ is a Wadge hierarchy and furthermore, if $M$ is a model of $\sf{AD}^+$ then $M$ is a model of Strong Mouse Capturing ($\sf{SMC}$). Operators that we construct in the core model induction will also have the following additional properties (besides being nice). 

In the following, a transitive structure $N$ is closed under an operator $\Omega$ if whenever $x\in \dom(\Omega)\cap N$, then $\Omega(x)\in N$.

\begin{definition}[relativizes well]\label{relativizeWell}
Let $\Omega$ be an operator (in the sense of \cite[Definition 3.20]{operator_mice}). We say that $\Omega$ $\rm{relativizes \ well}$ if there is a formula $\phi(x,y,z)$ such that for any $a,b\in \textrm{dom}(\Omega)$ such that $a\in L_1(b)$, whenever $N$ is a transitive model of $\sf{ZFC}^-$ such that $N$ is closed under $\Omega$, $a,b, \Fop(b)\in N$ then $\Fop(a)\in N$ and is the unique $x\in N$ such that $N\vDash \phi[x, a, \Fop(b)]$.
\end{definition}
\begin{definition}[determines itself on generic extensions]\label{detGenExts}
Suppose $\Omega$ is an operator. We say that $\Omega$ $\rm{determines \ itself \ on \ generic \ extensions}$ if there is a formula $\phi(x,y,z)$, a parameter $c$ such that for any transitive structure $N$ of $\sf{ZFC}^-$ such that $\omega_1\subset N$, $N$ contains $c$ and is closed under $\Omega$, for any generic extension $N[g]$ of $N$ in $V$, $\Omega\cap N[g]\in N[g]$ and is definable over $N[g]$ via $(\phi,c)$, i.e. for any $e\in N[g]\cap \dom(\Omega)$, $\Omega(e)=d$ if and only if $d$ is the unique $d'\in N[g]$ such that $N[g]\vDash \phi[c,d',e]$.
\end{definition}

To analyze $\Omega^*$, we adapt the framework for the core model induction developed above and the scales analysis in \cite{trang2012scales}, \cite{Scalesendgap}, and \cite{K(R)}. We are now in a position to introduce the core model induction operators that we will need in this paper. These are particular kinds of mouse operators (in the sense of \cite[3.43]{operator_mice}) that are constructed during the course of the core model induction. These operators can be shown to satisfy the sort of condensation described above, relativize well,  and determine themselves on generic extensions. 

Suppose $(\Omega,A)$ is nice ($\Omega$ can be a mouse operator or an iteration strategy).\footnote{From now on, we typically say: let $\Omega$ be a nice operator in place of this. So $\Omega$ is either a mouse operator in the sense of \cite{operator_mice} or an iteration strategy as in \cite{trang2012scales}.} Suppose $\Gamma$ is an inductive-like pointclass that is determined. Let $\MFsharp = \M_1^{X,\sharp}(A)$ where $X = (t_{(\Omega,A)})_2$; later on in the paper, we occasionally write $\M_1^{\Omega,\sharp}(A)$ for $\MFsharp$. Lp$^{^\g\Omega}(x)$ is defined as the stack of $^\g\Omega$-premice $\M$ over $x$ such that $\M$ is $x$-sound, there is some $n$ such that $\rho_{n+1}(\M)\leq o(x) < \rho_n(\M)$ and every countable, transitive $\M^*$ embeddable into $\M$ has an $(n,\omega_1+1)$-$^\g\Omega$-iteration strategy $\Delta$ for a coarse, transitive $x$.\footnote{Here $\rho_k(\M)$ denotes the $k$-th projectum of $\M$.} We define Lp$^{^\g\Omega,\Gamma}(x)$  similarly but demand additionally that $\Delta\in \Gamma$. For $\N$ a $^\g\Omega$-premouse, let Lp$^{^\g\Omega}
_+(\N)$ denotes the stack of all g-organized $\Omega$-premice
$\M$ such that either $\M = \N$, or $\N \lhd \M$, $\N$ is a strong cutpoint of
$\M$, $\M$ is $o(\N)$-sound, and there is $n <\omega$ such that $\rho_{n+1}(\M) \leq o(\N ) < \rho_n(M)$
and $\M$ is countably above-$o(\N)$, $Y$-$(n, \omega_1 + 1)$-iterable. We define Lp$_+^{^\g\Omega,\Gamma}
(\N)$ similarly. These notions can be generalized to $^\gTheta\Omega$ or any other operator in an obvious way (cf. \cite[Definition 2.43]{trang2012scales}).

\begin{definition}\label{dfn:C_Gamma}
Let $\Gamma$ be an inductive-like pointclass. For $x\in\mathbb{R}$, $C_\Gamma(x)$ denotes the set 
of all $y\in\mathbb{R}$ such that for some ordinal $\gamma<\omega_1$, $x$ (as a subset of $\omega$) is 
$\Delta_\Gamma(\{\gamma\})$.

Let $x\in\HC$ be such that $x$ is transitive and $f:\omega\to x$ a surjection. Then $c_f\in\mathbb{R}$ 
denotes the code for $(x,\in)$ determined by $f$.
And $C_\Gamma(x)$ denotes the set of all 
$y\in\HC\cap\powerset(x)$ such that for all surjections $f:\omega\to x$ we have $f^{-1}(y)\in 
C_\Gamma(c_f)$.
\end{definition}

\begin{definition}\label{dfn:k-suitable}
  Let $(\Omega,A)$ be as above, $t\in\HC$ with $\MFsharp\in\J_1(t)$. Let $1\leq k<\om$. A 
premouse $\N$ over $t$ is
\emph{$\Omega$-$\Gamma$-$k$-suitable} (or just \emph{$k$-suitable} if $\Gamma$ and $\Omega$ are clear from the context) iff there is a strictly increasing sequence
$\left<\delta_i\right>_{i<k}$ such that
\begin{enumerate}
 \item $\all\delta\in\N$, $\N\sats$``$\delta$ is Woodin'' if and only if $\ex 
i<k(\delta=\delta_i)$.
 \item $\OR(\N)=\sup_{i<\om}(\delta_{k-1}^{+i})^\N$.
\item\label{item:cutpoint} If $\N|\eta$ is a strong 
cutpoint of $\N$ then
$\N|(\eta^+)^\N=\Lp_+^{^\g\Omega,\Gamma}(\N|\eta)$.
 \item\label{item:Qstructure} Let $\xi<\OR(\N)$, where $\N\sats$``$\xi$ 
is 
not Woodin''. Then $C_\Gamma(\N|\xi)\sats$``$\xi$ is not Woodin''.
 \end{enumerate}
We write $\delta^\N_i=\delta_i$; also let $\delta_{-1}^\N=0$ and $\delta_k^\N=\OR(\N)$. \footnote{We could also define a suitable premouse $\N$ as a $\Theta$-g-organized $\Fop$-premouse and all the results that follow in this paper will be unaffected.}

\end{definition}

Let $\N$ be $1$-suitable and let $\xi\in\OR(\N)$ be a limit ordinal, such that 
$\N\sats$``$\xi$ isn't Woodin''. Let $Q\pins\N$ be the Q-structure for 
$\xi$. Let $\alpha$ be such that $\xi=\OR(\N|\alpha)$. If $\xi$ is a strong 
cutpoint of $\N$ then $Q\pins\Lp_+^{^\g\Omega,\Gamma}(\N|\xi)$ by \ref{item:cutpoint}.
Assume now that $\N$ is reasonably iterable. If $\xi$ is a strong cutpoint of 
$Q$, our mouse capturing hypothesis combined 
with \ref{item:Qstructure} gives that $Q\pins\Lp_+^{^\g\Omega,\Gamma}(\N|\xi)$. If $\xi$ 
is an $\N$-cardinal then indeed $\xi$ is a strong cutpoint of $Q$, since $\N$ 
has only finitely many Woodins. If $\xi$ is not a strong cutpoint of $Q$, then 
by definition, we do not have $Q\pins\Lp_+^{^\g\Omega,\Gamma}(\N|\xi)$. However, 
using 
$*$-translation (see \cite{DMATM}), one can find a level of 
$\Lp_+^{^\g\Omega,\Gamma}(\N|\xi)$ which corresponds to $Q$ (and this level is in $C_\Gamma(\N|\xi)$).

Suppose $\Omega$ is a nice operator and $\Sigma$ is an iteration strategy for a $\Omega$-$\Gamma$-$1$-suitable premouse $\P$ such that $\Sigma$ has branch condensation and is $\Gamma$-fullness preserving (for some pointclass $\Gamma$), then we say that $(\P,\Sigma)$ is a \textit{$\Omega$-$\Gamma$-suitable pair} or just \textit{$\Gamma$-suitable pair} or just \textit{suitable pair} if the pointclass and/or the operator $\Omega$ is clear from the context (this notion of suitability is not related to the one mentioned in Definition \ref{dfn:determines}).

The following definition gives examples of ``good operators''. This is not a standard definition and is given here for convenience more than anything. These are the kind of operators that the core model induction in this paper deals with. We by no means claim that these operators are all the useful model operators that one might consider. 

\begin{definition}[Core model induction operators]\label{cmi operator} \index{core model induction operators}Suppose $(\P, \Sigma)$ is a $\G$-$\Omega^*$-suitable pair for some nice operator $\G$ or a hod pair such that $\Sigma$ has branch condensation and is $\Omega$-fullness preserving. Let $\Omega = \Sigma$ (note that $\Omega$ is suitable). Assume $\rm{Code}(\Omega)$ is self-scaled. We say $J$ is a $\Sigma$-$\rm{core \ model \ induction \ operator}$ or just a $\Sigma$-$\rm{cmi\ operator}$ if one of the following holds:
\begin{enumerate}
 \item $J$ is a nice $\Omega$-mouse operator (or $\g$-organized $\Omega$-mouse operator) defined on a cone of $H_{\omega_1}$ above some $a\in H_{\omega_1}$. Furthermore, $J$ condenses finely, relativizes well and determines itself on generic extensions. 
 
\item  For some $\a\in \mathrm{OR}$ such that $\a$ ends either a weak or a strong gap in the sense of \cite{K(R)} and \cite{trang2012scales}, letting $M=\textrm{Lp}^{^\gTheta\Omega}(\mathbb{R},\rm{Code}(\Omega))||\a$ and $\Gamma = (\Sigma_1)^M$, $M\models \sf{AD}$$^++\sf{MC}$$(\Sigma)$.\footnote{\label{MC}$\textsf{MC}(\Sigma)$ stands for the Mouse Capturing relative to $\Sigma$ which says that for $x, y\in \mathbb{R}$, $x$ is $\mathrm{OD}(\Sigma, y)$ (or equivalently $x$ is $\mathrm{OD}(\Omega, y)$) iff $x$ is in some $^\g$-organized $\Omega$-mouse over $y$. $\sf{SMC}$ is the statement that for every hod pair $(\P,\Sigma)$ such that $\Sigma$ is fullness preserving and has branch condensation, then $\textsf{MC}(\Sigma)$ holds.} For some transitive $b\in H_{\omega_1}$ and some $1$-suitable (or more fully $\Omega$-$\Gamma$-$1$-suitable) $\Omega$-premouse $\Q$ over $b$, $J=\Lambda$, where $\Lambda$ is an $(\omega_1, \omega_1)$-iteration strategy for $\Q$ which is $\Gamma$-fullness preserving, has branch condensation and is guided by some self-justifying-system (sjs) $\vec{A}=(A_i: i<\omega)$ such that $\vec{A}\in OD_{b, \Sigma, x}^M$ for some real $x$ and $\vec{A}$ seals the gap that ends at $\alpha$\index{seal a gap}\footnote{This implies that $\vec{A}$ is Wadge cofinal in $\bf{Env}(\Gamma)$, where $\Gamma = \Sigma_1^{M}$. Note that $\bf{Env}(\Gamma) = \powerset(\mathbb{R})^M$ if $\alpha$ ends a weak gap and $\bf{Env}(\Gamma) = \powerset(\mathbb{R})^{\textrm{Lp}^\Sigma(\mathbb{R})|(\alpha+1)}$ if $\alpha$ ends a strong gap.}. 
\end{enumerate}
\end{definition}

\section{FROM $\Omega$ TO $\M_1^{\sharp,\Omega}$}\label{OneStep}

Suppose $(\P,\Sigma)$ is a $\G$-$\Omega^*$-suitable pair for some nice operator $\G$ such that $\Sigma$ has branch condensation and is $\Omega^*$-fullness preserving. (Recall that $\Omega^*$ is the pointclass of all sets of reals $A$ such that $L(A,\mathbb{R}) \models \mathsf{AD}^+$.)
As a special case we also allow $(\P,\Sigma) = (\emptyset, \emptyset)$; the analysis of this special case is enough to prove Theorem \ref{theorem:AD-equiconsistency}.
In this section we assume the strong hypothesis
\begin{quote}
 $\mathsf{ZF} + \mathsf{DC}_{\powerset(\omega_1)} + {}$``$\omega_1$ is $\mathbb{R}$-strongly compact and $\neg\square_{\omega_1}$.''
\end{quote}
Note that this follows from any of the hypotheses of Theorems \ref{theorem:AD-equiconsistency}, \ref{theorem:ADR-equiconsistency}, and \ref{theorem:ADR-DC-equiconsistency}.

Let $\Omega$ be a $\Sigma$-CMI operator.
(If $(\P,\Sigma) = (\emptyset, \emptyset)$ then $\Omega$ is an ordinary CMI operator of the kind typically used in proving $\mathsf{AD}^{L(\mathbb{R})}$.)
We will use our strong hypothesis to
obtain the $\mathcal{M}_1^{\Omega,\sharp}$ operator, which is the relativization of the $\mathcal{M}_1^{\sharp}$ operator to a fine-structural hierarchy where the levels are obtained by repeated applications of the $\Omega$ operator (rather than the $\text{rud}$ operator, as in ordinary mice. Basically, for each $x$ in dom$(\Omega)$, if $\Omega$ is a strategy, $\M_1^{\Omega,\sharp}(x)$ is $\M_1^{X,\sharp}(x)$, where $X=(\Omega,\varphi_{\rm{min}})$ and $\varphi_{\rm{min}}$ is defined as in \cite[Definition 3.2]{trang2012scales} and otherwise $\M_1^{\Omega,\sharp}(x)$ is defined as in \cite{operator_mice}.)

The argument is similar to that used to obtain the ordinary $\mathcal{M}_1^{\sharp}$ operator from the failure of square at a measurable cardinal in $\mathsf{ZFC}$.
The relativization of the standard arguments from $\mathcal{M}_1^{\sharp}$ to $\mathcal{M}_1^{\Omega,\sharp}$ presents no special problems, but working without the Axiom of Choice requires a bit of care because ultrapowers of $V$
may fail to satisfy \Los's theorem.
However, \Los's theorem does hold for ultrapowers of wellordered inner models of $V$, and more generally for ultraproducts of families of inner models that are uniformly wellordered in the sense that there is 
a function associating to each model a wellordering of that model.

The relevance of Jensen's square principle $\square_\kappa$ is that it
holds for all infinite cardinals $\kappa$ in
all Mitchell--Steel extender models (mice)
by Schimmerling and Zeman \cite[Theorem 2]{schimmerling2001square}.
The proof of this result is sufficiently abstract that it relativizes from mice to $\Omega$-mice in a straightforward manner.
Therefore if $\square_\kappa$ fails in $V$, we get a failure of covering:
the successor of $\kappa$ cannot be computed correctly by any $\Omega$-mouse.

Because we are not assuming the Axiom of Choice, we will not construct the core model in $V$ but rather in an inner model $H$ of $V$ satisfying $\mathsf{ZFC}$.  This model $H$ will be obtained as a kind of $\HOD$.
A method used by  Schindler and Steel \cite{MaxCM} to prove covering results for the core model
of $V$ can be adapted to the core model of $H$,
provided that we can show that $H$ is close enough to $V$ in the relevant sense.
We show this closeness by using Vop\v{e}nka's theorem, similar to 
Schindler \cite{schindler1999successive}.

The following lemma is the main result of this section.  It will form the ``successor step'' in the proofs of the main theorems.

\begin{lemma}\label{lemma-not-square-implies-M1Fsharp}
 Assume $\mathsf{ZF} + \mathsf{DC}_{\powerset(\omega_1)} + {}$``$\omega_1$ is $\mathbb{R}$-strongly compact and $\neg\square_{\omega_1}$.''
 Let $(\P,\Sigma)$ be a $\G$-$\Omega^*$-suitable pair for some nice operator $\G$, a hod pair such that $\Sigma$ has branch condensation and is $\Omega^*$-fullness preserving, or $(\emptyset,\emptyset)$.
 Let $\Omega$ be a $\Sigma$-CMI operator defined on a cone in $H(\omega_1)$ over some element $a \in H(\omega_1)$.
 Then for every element $x$ of this cone, $\mathcal{M}_1^{\Omega,\sharp}(x)$ exists.
\end{lemma}
\begin{proof}
 First, note that we may assume without loss of generality that full $\mathsf{DC}$ holds, by passing to the inner model $L(\powerset(\omega_1),\Sigma,\Omega)[\mu]$ where we are constructing relative to a predicate $\mu$ for a fine countable complete measure on $\powerset_{\omega_1}(\mathbb{R})$.
 The hypothesis and conclusion are absolute to this inner model. In particular the model satisfies $\neg\square_{\omega_1}$ because it computes $\omega_2$ correctly, and it satisfies $\mathsf{DC}_{\powerset(\omega_1)}$ because it contains all countable sequences from $\powerset(\omega_1)$. In the inner model, this fragment of $\mathsf{DC}$ implies full $\mathsf{DC}$ by a standard argument using the fact that every set is the surjective image of $\powerset(\omega_1) \times \alpha$ for some ordinal $\alpha$.
 Therefore we may safely use $\mathsf{DC}$ in the argument that follows.
 
 Note that because $\omega_1$ is measurable, the operators $\Omega^\sharp$ and $\Omega^{\sharp^\sharp}$ are also defined on the cone in $H(\omega_1)$ over $a$.
 Let $x \in H(\omega_1)$ be in the cone over $a$.
 Take a countably complete fine measure $\mu$ on $\powerset_{\omega_1}(\mathbb{R})$.
 For $\mu$-almost every set $\sigma$ we have $x \in \sigma$ and we can define the inner model
 \[H_\sigma = \HOD_{\{\Omega,x\}}^{L^{\Omega^{\sharp}}(\sigma)}.\]
 
 A few remarks on notation:
 The model $L^{\Omega^{\sharp}}(\sigma)$ is the proper class model that is obtained by iterating
 the top measure of $\Omega^{\sharp^\sharp}(\sigma)$ out of the universe.
 It is closed under its version of $\Omega$ even above the point $\omega_1^V$ up to which $\Omega$ was originally defined; however, we will only ever use the $\Omega$ operator of the model $L^{\Omega^{\sharp}}(\sigma)$ up to the least indiscernible of that model, which is the critical point of the top measure of $\Omega^{\sharp^\sharp}(\sigma)$ and is countable in $V$. 
 By the parameter $\Omega$ in the definition of $H_\sigma$, we really mean the restriction of $\Omega$ to the model $L^{\Omega^{\sharp}}(\sigma)$, which is amenable to that model because $\Omega$ relativizes well. There will not be any incompatibility between the various restrictions and extensions of $\Omega$ that we use, so we denote them all by ``$\Omega$''.
  
 Let $\xi_\sigma$ denote the least indiscernible of $L^{\Omega^{\sharp}}(\sigma)$.
 Note that in the model $H_\sigma$ we can do core model theory below $\xi_\sigma$: it is well-known that the existence of an external measure can substitute for measurability of $\xi_\sigma$ in this regard.
 The operator $\Omega$ is amenable to $H_\sigma$ (again because it relativizes well)
 and we can attempt the $K^{c,\Omega}(x)$ construction in $H_\sigma$ up to the cardinal $\xi_\sigma$.
 This is like the ordinary $K^c$ construction, except relativized to $\Omega$ and built over the set $x$ (see \cite[Definition 3.28]{operator_mice} and \cite[Definition 2.46]{trang2012scales}).
 By the $K^\Omega$ existence dichotomy (see Schindler and Steel \cite{CMI}) applied in the various models $H_\sigma$, one of the following two cases holds:
 
 \begin{enumerate}
  \item\label{item-M-1-F-sharp} For $\mu$-almost every set $\sigma \in \powerset_{\omega_1}(\mathbb{R})$, the model $H_\sigma$ satisfies the statement that $\mathcal{M}_1^{\Omega,\sharp}(x)$ exists and is $\xi_\sigma$-iterable by the (unique) $\Omega^\sharp$-guided strategy.
  \item\label{item-K} For $\mu$-almost every set $\sigma \in \powerset_{\omega_1}(\mathbb{R})$, the model $K_\sigma$, defined as the core model $(K^{\Omega}(x))^{H_\sigma}$ built up to $\xi_\sigma$, exists and has no Woodin cardinals.
 \end{enumerate}

 \begin{claim}
  If case \eqref{item-M-1-F-sharp} of the $K^\Omega$ existence dichotomy holds, then $\mathcal{M}_1^{\Omega,\sharp}(x)$ exists in $V$.
 \end{claim}
 \begin{proof}
   For $\mu$-almost every set $\sigma \in \powerset_{\omega_1}(\mathbb{R})$, the
  premouse $(\mathcal{M}_1^{\Omega,\sharp}(x))^{H_\sigma}$ exists by the case hypothesis.
  It is sound and projects to $x$, so it codes itself as a subset of $x$, which is countable.
  Therefore by the countable completeness of $\mu$ we can fix a single $\Omega$-premouse $\mathcal{M}$ over $x$ such that $\mathcal{M} = (\mathcal{M}_1^{\Omega,\sharp}(x))^{H_\sigma}$ for $\mu$-almost every set $\sigma$.
  We will show that $\mathcal{M}$ is $\omega_1$-iterable in $V$ by the (unique) $\Omega^\sharp$-guided iteration strategy.  Then $(\omega_1+1)$-iterablity will follow by the measurability of $\omega_1$.
  
  Let $\mathcal{T}$ be a countable  $\Omega^\sharp$-guided putative iteration tree on $\mathcal{M}$ in $V$,
  where by ``putative'' we mean that its last model, if it has one, may fail to be an $\Omega$-premouse. (Note that an $\Omega$-premouse is required in particular to be wellfounded, and this is the only requirement if $\Omega = \operatorname{rud}$.)
  We want to show that if $\mathcal{T}$ has successor length, then its
  last model is an $\Omega$-premouse, and if it has limit length, then it has a
  cofinal branch $b$ such that $\mathcal{M}^\mathcal{T}_b$ is an $\Omega$-premouse and $\mathcal{Q}(b,\mathcal{T}) \unlhd \Omega^\sharp(\mathcal{M}(\mathcal{T}))$.
  
  Take a real $t$ that codes $\mathcal{T}$.
  Then for $\mu$-almost every set $\sigma$ we have $t \in \sigma$
  by the fineness of $\mu$.
  Fix a set $\sigma$ such that
  $H_\sigma$ satisfies the statement ``$\mathcal{M}_1^{\Omega,\sharp}(x)$ exists and is $\xi_\sigma$-iterable,'' $(\mathcal{M}_1^{\Omega,\sharp}(x))^{H_\sigma} = \mathcal{M}$, and $t \in \sigma$.
  By Vopenka's theorem applied in the model $L^{\Omega^{\sharp}}(\sigma)$,
  the real $t$ is contained in a generic extension $H_\sigma[g]$ of $H_\sigma$.
  In fact because $\xi_\sigma$ is inaccessible in $L^{\Omega^{\sharp}}(\sigma)$ 
  the poset from the proof of Vopenka's theorem (see, for example, Jech \cite[Theorem\ 15.46]{Jech}) 
  is in $(V_{\xi_\sigma})^{H_\sigma}$.

  In $H_\sigma$ the $\Omega$-premouse $\mathcal{M}$ is $\xi_\sigma$-iterable by the $\Omega^\sharp$-guided strategy, by our assumptions.
  Because the $\Omega^\sharp$ operator condenses well and 
  determines itself on generic extensions in the sense of Schindler and Steel \cite[Definition 1.4.10]{CMI},\footnote{In the ``gap in scales'' case, the proof that the $\Omega^\sharp$ operator determines itself on generic extensions is given by Schindler and Steel \cite[Section 5.6, proof of Claim 1 in case $n=0$]{CMI}.  The proof in the other cases is a straightforward induction.}
  a standard argument (see Schindler and Steel \cite[Lemma 2.7.2]{CMI}) shows that
  $\mathcal{M}$ is still $\xi_\sigma$-iterable in $H_\sigma[g]$ by
  the $\Omega^\sharp$-guided iteration strategy there.

  The model $H_\sigma[g]$ sees that the tree $\mathcal{T}$ is $\Omega^\sharp$-guided. Therefore in $H_\sigma[g]$,
  if $\mathcal{T}$ has successor length, then the last model of $\mathcal{T}$ is a wellfounded $\Omega$-premouse, and if $\mathcal{T}$ has limit length, then it has a cofinal branch $b$ such that $\mathcal{M}_b^\mathcal{T}$ is an $\Omega$-premouse and
  $\mathcal{Q}(b,\mathcal{T}) \unlhd \Omega^\sharp(\mathcal{M}(\mathcal{T}))$.
  In either case this fact about $\mathcal{T}$ is absolute to $V$, giving the desired iterability.
 \end{proof}
 
 \begin{claim}
  Case \eqref{item-K} of the $K^\Omega$ existence dichotomy cannot hold.  
 \end{claim}
 \begin{proof}
 This case is where the hypothesis $\neg \square_{\omega_1}$ is used.
 Because $H_\sigma$ is defined as the $\HOD_{\{\Omega,x\}}$ of $L^{\Omega^{\sharp}}(\sigma)$,
 we can define the Vop\v{e}nka poset $\mathbb{P}_\sigma \in H_\sigma$ to make 
 every countable set of countable ordinals in $L^{\Omega^{\sharp}}(\sigma)$ generic over $H_\sigma$.
 For a countable set of countable ordinals $a$ of $L^{\Omega^{\sharp}}(\sigma)$,
 let $g_{\sigma,a}$ denote the $H_\sigma$-generic filter over $\mathbb{P}_\sigma$ induced by $a$, which has the property that $a \in H_\sigma[g_{\sigma,a}]$.\footnote{Unlike in case \eqref{item-M-1-F-sharp}, it is important here that the Vop\v{e}nka generic filter $g_{\sigma,a}$ is induced by $a$ itself and does not depend on the choice of a real coding $a$.}
 Note that $H_\sigma \in (V_{\xi_\sigma})^{H_\sigma}$ because $\xi_\sigma$ is inaccessible in $L^{\Omega^{\sharp}}(\sigma)$.

 Define the ultraproducts
 \begin{align*}
  H &= [\sigma \mapsto H_\sigma]_\mu &   \Xi &= [\sigma \mapsto \xi_\sigma]_\mu \\
  K &= [\sigma \mapsto K_\sigma]_\mu &   \mathbb{P} &= [\sigma \mapsto \mathbb{P}_\sigma]_\mu.
 \end{align*}
 Every countable set of countable ordinals $a$ in $V$
 is seen as a countable set of countable ordinals in $L^{\Omega^{\sharp}}(\sigma)$
 for $\mu$-almost every $\sigma$ (by fineness applied to a real coding $a$)
 so  we can define the ultraproduct
 \[ g_a = [\sigma \mapsto g_{\sigma,a}]_\mu.\]
 Then applying \Los's theorem to uniformly wellordered families of structures is enough to establish the following facts.\footnote{If the measure $\mu$ were normal, then \Los's theorem could be applied to the models $L^{\Omega^{\sharp}}(\sigma)$ themselves to yield a model $L^{\Omega^{\sharp}}(\mathbb{R})$ in which $H$, $K$, $\Xi$, and $\mathbb{P}$ could then be defined.  But this is not possible in general, for example under $\mathsf{AD} + V=L(\mathbb{R})$, where the hypothesis of the lemma holds for $\Omega = \operatorname{rud}$ but $\mathbb{R}^\sharp$ does not exist.}
 \begin{itemize}
  \item $H$ is an inner model of $\mathsf{ZFC}$ with a cardinal $\Xi > \omega_1^V$ that is large enough to do core model theory below it.
  \item $K$ is the core model of $H$ built up to $\Xi$, and it has no Woodin cardinals.
  \item $\mathbb{P} \in (V_\Xi)^H$ is a forcing poset.
  \item To each countable set of countable ordinals $a$ in $V$ we have assigned an $H$-generic filter $g_a \subset \mathbb{P}$ such that $a \in H[g_a]$.
 \end{itemize}

 Now write $\kappa = \omega_1^V$ and define the $\mu$-ultrapower map
 \[j : V \to \Ult(V,\mu),\quad \crit(j) = \kappa.\]
 Recall that $j$ itself is not elementary, but its restrictions to wellordered inner models are elementary.
 (We remark that one could use any ultrapower map with critical point $\kappa$ here; the measurability of $\omega_1^V$ suffices for the following argument in place of $\mathbb{R}$-strong compactness of $\omega_1^V$, although it is not clear that it would suffice for the previous argument.)
 
 Note that to every set $A \subset \kappa$ in $V$ we
 can assign a $j(H)$-generic filter $g_A \subset j(\mathbb{P})$ such that $A \in j(H)[g_A]$.
 To see this, consider the sequence of generic filters $\vec{g}_A = (g_{A \cap \alpha} : \alpha < \kappa)$, use the elementarity of the map 
 $j \restriction L[H,A,\vec{g}]$, and define $g_A = j(\vec{g}_A)_\kappa$.

 Because $\square_\kappa$ fails in $V$, we have 
 \[(\kappa^+)^{j(K)} < (\kappa^+)^V\]
 by a result of Schimmerling and Zeman \cite[Theorem 2]{schimmerling2001square} relativized to the operator $j(\Omega)$ and applied to the model $j(K)$, which is the core model of $j(H)$.
 
 Take a set $A \subset \kappa$ in $V$ coding a wellordering of $\kappa$ of order type $(\kappa^+)^{j(K)}$
 and define $g = g_A$.
 Because $A \in j(H)[g]$ we get
 \[ j(H)[g] \models (\kappa^+)^{j(K)} < \kappa^+.\]
 Because $g$ was added by a small forcing below the large cardinal $j(\Xi)$ where $j(K)$ was constructed, we have that $j(K)$ is still the core model of $j(H)[g]$.\footnote{To make sense of the core model of $j(H)[g]$ we are using the fact that $j(H)$'s version of the operator $\Omega$ determines itself on generic extensions.}
 Therefore (and this is the crucial point) the model $j(H)[g]$ sees the failure of covering for its own core model at $\kappa$, so we can apply the map $j$ once more to get a contradiction by a standard argument, outlined below.

 Consider the restriction
 \[j \restriction j(H)[g] : j(H)[g] \to j(j(H))[j(g)],\]
 which is an elementary embedding.
 Because the domain $j(H)[g]$ satisfies $(\kappa^+)^{j(K)} < \kappa^+$, the further restriction $j \restriction \powerset(\kappa)^{j(K)}$ is in the codomain $j(j(H))[j(g)]$
 by a standard argument due to Kunen.
 Therefore
 we have 
 \[F \in j(j(H))[j(g)]\]
 where $F$ is the
 $(\kappa,j(\kappa))$-extender over $j(K)$ derived from the map $j \restriction \powerset(\kappa)^{j(K)}$.
 Note that $K|\kappa = j(K)|\kappa$, and $\kappa$ is an inaccessible cardinal in both $\mathsf{ZFC}$ models $K$ and $j(K)$ because it is a measurable cardinal in $V$.  Therefore $j(K)|j(\kappa) = j(j(K))|j(\kappa)$, and $j(\kappa)$ is an inaccessible cardinal in both models $j(K)$ and $j(j(K))$, so we have
 \begin{align*}
  (\kappa^+)^{j(K)} &= (\kappa^+)^{j(j(K))} < j(\kappa) \quad \text{and}\\
  \powerset(\kappa)^{j(K)} &= \powerset(\kappa)^{j(j(K))}.
 \end{align*}
 Therefore the extender $F$
 can also be considered as an extender over $j(j(K))$,
 and it coheres with $j(j(K))$.
 Note that $j(j(K))$ is the core model of $j(j(H))[j(g)]$
 
 This extender $F$ has superstrong type, and we can apply
 the maximality property of the core model \cite[Theorem 2.3]{MaxCM} in the model $j(j(H))[j(g)]$
 to show that every proper initial segment $F \restriction \nu$ of $F$, where $\nu < j(\kappa)$, is on the sequence of the core model $j(j(K))$.
 Then in the core model $j(j(K))$, these initial segments will witness that $\kappa$ is a Shelah cardinal.
 This will contradict our case hypothesis, which says that there are no Woodin cardinals in $K$.
 
 Let $\mathcal{M} = j(j(K))$ and let $F \restriction \nu$, where $\nu < j(\kappa)$, be a proper initial segment of $F$.
 We want to see that $F \restriction \nu$ is on the sequence of the core model $j(j(K))$.
 Without loss of generality we may assume that $\nu$ is at least the common $\kappa^+$ of the models $j(K)$ and $\mathcal{M}$.
 It suffices to show that
 the pair $(\mathcal{M}, F \restriction \nu)$ is weakly countably certified \cite[Definition 2.2]{MaxCM}.
 Working in the model $j(H)[g]$, take a transitive, power admissible set $N$ such that $N^\omega \subset N$,
 $V_\kappa \cup j(K)|((\kappa^+)^{j(K)}+1) \subset N$, and $\left|N\right| = \kappa$.
 Stepping out to $V$ for a moment and applying Kunen's argument again, we have
 \[G \in j(j(H))[j(g)]\]
 where $G$ is the
 $(\kappa,j(\kappa))$-extender over $N$ derived from $j \restriction \powerset(\kappa)^N$.
 Now in the model $j(j(H))[j(g)]$ it is easy to verify that the pair
 $(N,G)$ is a weak $\mathcal{A}$-certificate \cite[Definition 2.1]{MaxCM}
 for $(\mathcal{M},F \restriction \nu)$
 whenever $\mathcal{A}$ is a countable subset of $\bigcup_{n<\omega} \powerset([\kappa]^n) \cap \mathcal{M}|\nu$,\footnote{Or indeed if $\mathcal{A}$ is equal to $\bigcup_{n<\omega} \powerset([\kappa]^n) \cap \mathcal{M}|\nu$ itself; we don't need countability, and we don't need to choose the certificate $(N,G)$ differently depending on $\mathcal{A}$ (or on $\nu$, for that matter.)}
 noting that $\mathcal{M}|\nu$, $\mathcal{M}$, and $j(K)$ all have the same subsets of $[\kappa]^n$ (because $\nu$ is greater than or equal to the common $\kappa^+$ of $j(K)$ and $\mathcal{M}$.)
\end{proof}

We have shown that if case \eqref{item-M-1-F-sharp} of the $K^\Omega$ existence dichotomy holds, then the conclusion of the lemma holds, and we have shown that case \eqref{item-K} contradicts the hypothesis of the lemma, so the proof of the lemma is complete.
\end{proof}

We remark that because $\Omega$ is a $\Sigma$-CMI operator, the operator  $\mathcal{M}_1^{\Omega,\sharp}$ given by the lemma is also a $\Sigma$-CMI operator.

\begin{corollary}\label{corollary:not-square-implies-PD}
 Assume $\mathsf{ZF} + \mathsf{DC}_{\powerset(\omega_1)} + {}$``$\omega_1$ is $\mathbb{R}$-strongly compact and $\neg\square_{\omega_1}$.'' Then $\mathsf{PD}$ holds.
\end{corollary}
\begin{proof}
 We show by induction on $n < \omega$ that the $\mathcal{M}_n^\sharp$ operator is total on $H(\omega_1)$.
 The base case is the $\mathcal{M}_0^\sharp$ operator, meaning the ordinary sharp operator, which is total on $H(\omega_1)$ because $\omega_1$ is measurable.
 For the induction step we apply Lemma \ref{lemma-not-square-implies-M1Fsharp}
 to go from the operator $\Omega = \mathcal{M}_n^\sharp$ to the operator $\mathcal{M}_1^{\Omega,\sharp}$, which is stronger than $\mathcal{M}_{n+1}^\sharp$.
 It follows from the existence of $\mathcal{M}_n^\sharp(x)$ for every $n < \omega$ and $x \in \mathbb{R}$
 that Projective Determinacy holds.
\end{proof}

In the next section we will strengthen this conclusion to $\mathsf{AD}^{L(\mathbb{R})}$ and thereby obtain an equiconsistency result (Theorem \ref{theorem:AD-equiconsistency}.)
\section{THE MAXIMAL MODEL OF ``$\sf{AD}$$^+ + \Theta=\theta_\Sigma$"}\label{CMIInsideLp}

Throughout this section, we assume the hypothesis of Lemma \ref{lemma-not-square-implies-M1Fsharp}, namely we assume

\begin{quote}
 $\mathsf{ZF} + \mathsf{DC}_{\powerset(\omega_1)} + {}$``$\omega_1$ is $\mathbb{R}$-strongly compact and $\neg\square_{\omega_1}$.''
\end{quote}
Suppose $(\P,\Sigma)$ is a $\G$-$\Omega^*$-suitable pair for some nice operator $\G$ such that $\Sigma$ has branch condensation and is $\Omega^*$-fullness preserving.
As a special case we also allow $(\P,\Sigma) = (\emptyset, \emptyset)$; the analysis of this special case is enough to prove Theorem \ref{theorem:AD-equiconsistency}.
We first define the ``maximal pointclass of $\mathsf{AD}^+ + \Theta = \theta_\Sigma$".
\begin{definition}\label{def:max_sigma}
Let $(\P,\Sigma)$ be as above. Let
\[\Omega_\Sigma = \bigcup\{\powerset(\mathbb{R})\cap L(A,\mathbb{R}) \mid A\subseteq \mathbb{R},  L(A,\mathbb{R})\vDash \mathsf{AD}^+ + \Theta = \theta_\Sigma + \mathsf{MC}(\Sigma)\}.\]
\end{definition}

We note that by $(\dag)$, $\Omega_\Sigma$ is a Wadge hierarchy. In the case $(\P,\Sigma)= (\emptyset,\emptyset)$,
substitute $\theta_0$ for $\theta_\Sigma$ and ordinary mouse capturing $\mathsf{MC}$ for $\mathsf{MC}(\Sigma)$.
In this section, we will prove that 
\begin{equation}\label{eqn:maximal_model}
L(\Omega_\Sigma,\mathbb{R})\cap \powerset(\mathbb{R}) = \Omega_\Sigma.
\end{equation}
This has the consequence that $L(\Omega_\Sigma,\mathbb{R}) \vDash \sf{AD}$$^+ + \Theta=\theta_\Sigma$. The model $L(\Omega_\Sigma,\mathbb{R})$ is called the ``maximal model of $\sf{AD}$$^+ + \Theta=\theta_\Sigma$".

Let $\Omega=\Sigma$. The proof of \eqref{eqn:maximal_model} depends on understanding models of $\mathsf{ZF} + \mathsf{AD}^+ + V = L(\powerset(\mathbb{R})) + \Theta = \theta_\Sigma + \mathsf{MC}(\Sigma)$ as hybrid mice over $\mathbb{R}$, $\Theta$-$\g$-organized as in Section \ref{g_organized}. (In the case $(\mathcal{P},\Sigma) = (\emptyset,\emptyset)$, we consider ordinary mice over $\mathbb{R}$, namely levels of $\Lp(\mathbb{R})$, and we do not need $\Theta$-$\g$-organization by Remark \ref{SamePR}. To keep the notations uniform in this section, we will use the notation Lp$^{^\gTheta\Omega}(\mathbb{R},\rm{Code}(\Omega))$ to denote Lp$(\mathbb{R})$ in the case $(\P,\Sigma)=(\emptyset,\emptyset)$.)

$\Omega$ is suitable and $\M_1^{\Omega,\sharp}$ generically interprets $\Omega$. \footnote{By results of \cite{trang2012scales}, $\M_1^{\Omega,\sharp}$ generically interprets $\Omega$ for $(\P,\Sigma)$ being a $\mathcal{G}$-$\Omega$-suitable pair or a hod pair with $\Sigma$ having branch condensation and is $\Omega$-fullness preserving.} Again, we suppress parameter $A$ and let $X$ be defined from $(\Omega,A)$ as before. 
Let $\Lambda$ be the unique $\omega_1+1$-$\Omega$-iteration strategy for $\M_1^{\Omega,\sharp}$.
It can be shown to follow from the hypotheses of Theorems \ref{theorem:ADR-equiconsistency} and \ref{theorem:ADR-DC-equiconsistency}
(in particular using the fact that every uncountable regular cardinal $\le \Theta$ is threadable) that
the iteration strategy $\Lambda$ can be extended to a unique $\Theta+1$-iteration strategy with branch condensation, which we will also call $\Lambda$. (This ``strategy extension'' step is not necessary for the case $(\mathcal{P},\Sigma) = (\emptyset,\emptyset)$, so we postpone its proof until Section \ref{section:Theta-gt-theta-Sigma}.)

As in \cite{trang2012scales}, we use $\Lambda$ to define Lp$^{^\gTheta\Omega}(\mathbb{R},\rm{Code}(\Omega))$. The only thing to check is that $\Theta+1$-iterability is sufficient to run the definition of Lp$^{^\gTheta\Omega}(\mathbb{R},\rm{Code}(\Omega))$ in \cite{trang2012scales}. Suppose by induction, we have defined a level $\M \lhd \rm{Lp}^{^\gTheta\Omega}(\mathbb{R},\rm{Code}(\Omega))$ (in general, the following argument works for any transitive structure $M$ containing $\mathbb{R}$ such that there is a surjection from $\mathbb{R}$ onto $M$) and without loss of generality, we assume $\M$ is a tree activation level $\N_{\alpha+1}$ and we are trying to define the level $\M_{\alpha+1}$ (in the notation of \cite[Definition 3.38]{trang2012scales}); this just means that $\M_{\alpha+1}$ is the first level above $\M$ by which we have fed in all necessary branch information about $\T_\M$. It comes down to defining $\T_\M$ as in Definition \ref{genGenTree}. Working in the model $N = L(\M,\mathbb{R}, f)[\Sigma]$ \footnote{By ``$\Sigma$", we mean the set $\{(\T,\beta) : \beta \in \Sigma(\T)\}$.}, where $f$ is a surjection from $\mathbb{R}$ onto $\M$, we need to see that the genericity iteration that defines $\T_\M$ terminates in less than $\Theta$ many steps. Suppose not, letting $\T\in N$ be the corresponding tree of length $\Theta+1$. In $N$, letting $\gamma$ be a large regular cardinal $>\Theta$, we can construct some $X\prec L_{\gamma}(\M,\mathbb{R},f)[\Sigma]$ that contains all relevant objects (in particular, $\mathbb{R}\cup\M\cup\{\M\}\subset X$) and there is a surjection from $\mathbb{R}$ onto $X$. Let $\pi: M_X\rightarrow X$ be the uncollapse map and let $\xi = \rm{crt}(\pi)$; then $\xi < \Theta$ and $\pi(\xi)\leq \Theta$. We note that $\pi$ can be canonically extended to a map $\pi^+: M_X[G]\rightarrow L_{\gamma}(\M,\mathbb{R},f)[\Sigma][G]$, where $G\subseteq Col(\omega,\mathbb{R})$ is $L(\M,\mathbb{R},f)[\Sigma]$-generic.  We also note that since $\M\cup\{\M\}\subset X$, $\xi > o(\M)$. We can then use standard arguments (cf. \cite[Theorem 3.11]{steel2010outline}), where $X[G]$ plays the role of the countable hull $X$ there, to conclude that lh$(\T)<\Theta$. Contradiction. So $\T_\M$ is defined and has length $< \Theta$.

To prove \eqref{eqn:maximal_model}, we need the following definition.
\begin{definition}\label{def:sLp}
We define $\rm{sLp}^{^\gTheta\Omega}(\mathbb{R},\rm{Code}(\Omega))$ to be the union of those $\M\lhd \rm{Lp}^{^\gTheta\Omega}(\mathbb{R},\rm{Code}(\Omega))$ such that whenever $\pi:\M^*\rightarrow \M$ is elementary, $\P\in \pi^{-1}(\rm{HC})$, and $\M^*$ is countable, transitive, then $\M^*$ is $X$-$\omega_1+1$-iterable with unique strategy $\Lambda$ such that $\Lambda\rest \HC \in \M$.
\end{definition}

We note that $\rm{sLp}^{^\gTheta\Omega}(\mathbb{R},\rm{Code}(\Omega))$ is an initial segment of $\rm{Lp}^{^\gTheta\Omega}(\mathbb{R},\rm{Code}(\Omega))$ \footnote{The initial segment may be strict.}and $\rm{sLp}^{^\gTheta\Omega}(\mathbb{R},\rm{Code}(\Omega))$ is trivially constructibly closed. Also, $\rm{sLp}^{^\gTheta\Omega}(\mathbb{R},\rm{Code}(\Omega))\vDash \Theta =\theta_\Sigma$ and the extender sequence of $\rm{sLp}^{^\gTheta\Omega}(\mathbb{R},\rm{Code}(\Omega))$ is definable over $\rm{sLp}^{^\gTheta\Omega}(\mathbb{R},\rm{Code}(\Omega))$ from $\Omega$, which in turns is definable from $\Sigma$. In this section, we outline the core model induction up to the ``last gap" of $\rm{sLp}^{^\gTheta\Omega}(\mathbb{R},\rm{Code}(\Omega))$. This will show that 
\begin{equation}\label{eqn:sLp_is_maximal_model}
\rm{sLp}^{^\gTheta\Omega}(\mathbb{R},\rm{Code}(\Omega)) \vDash \sf{AD}^+ + \textsf{MC}(\rm{\Sigma}). \footnote{Ordinal definability from $\Sigma$ in the definition of $\sf{MC}(\Sigma)$ is in the language of set theory, not in the language of $\rm{sLp}^{^\gTheta\Omega}(\mathbb{R},\rm{Code}(\Omega))$, but by the paragraph above \ref{def:sLp}, this will not make a difference.}
\end{equation}
From \cite[Theorem 17.1]{DMATM} and \cite{sargsyan2014Rmice}, we know that if $M\vDash ``V = L(\powerset(\mathbb{R})) + \sf{AD}$$^+ + \MC(\Sigma) + \Theta=\theta_\Sigma$", then $M \vDash ``V= L(\rm{sLp}$$^{^\gTheta\Omega}(\mathbb{R},\rm{Code}(\Omega)))$". This and equation \ref{eqn:sLp_is_maximal_model} imply equation \ref{eqn:maximal_model}. It then suffices to prove equation \ref{eqn:sLp_is_maximal_model}. The rest of the section is devoted to this task.

The following definitions are obvious generalizations of those defined in \cite{CMI}.

\begin{definition}
We say that the coarse mouse witness condition $W^{*,^\g\Omega}_\gamma$ holds if, whenever $U\subseteq \mathbb{R}$ and both $U$ and its complement have scales in $\rm{Lp}$$^{^\gTheta\Omega}(\mathbb{R},\rm{Code}(\Omega))|\gamma$, then for all $k< \omega$ and $x \in \mathbb{R}$ there is a coarse $(k,U)$-Woodin $^\g\Omega$-mouse\footnote{This is the same as the usual notion of a $(k,U)$-Woodin mouse, except that we demand the mouse is a $\g$-organized $\Omega$-mouse.} containing $x$ with an $(\omega_1 + 1)$-iteration $^\g\Omega$-strategy\footnote{In our context, where $\omega_1$ is measurable, this is equivalent to $\omega_1$-iterability.} whose restriction to $H_{\omega_1}$ is in $\rm{Lp}$$^{^\gTheta\Omega}(\mathbb{R},\rm{Code}(\Omega))|\gamma$.
\end{definition}
\begin{remark}
By the proof of \cite[Lemma 3.3.5]{CMI}, $W^{*,^\g\Omega}_\gamma$ implies $\rm{Lp}$$^{^\gTheta\Omega}(\mathbb{R},\rm{Code}(\Omega))|\gamma \vDash \sf{AD}$.
\end{remark}
\begin{definition}
An ordinal $\gamma$ is a \textbf{critical ordinal} in $\rm{Lp}$$^{^\gTheta\Omega}(\mathbb{R},\rm{Code}(\Omega))$ if there is some $U \subseteq \mathbb{R}$
such that $U$ and $\mathbb{R} \backslash U$ have scales in $\rm{Lp}$$^{^\gTheta\Omega}(\mathbb{R},\rm{Code}(\Omega))|(\gamma + 1)$ but not in $\rm{Lp}$$^{^\gTheta\Omega}(\mathbb{R},\rm{Code}(\Omega))|\gamma$. In other words, $\gamma$
is critical in $\rm{Lp}$$^{^\gTheta\Omega}(\mathbb{R},\rm{Code}(\Omega))$ just in case $W^{*,^\g\Omega}_{\gamma+1}$ does not follow trivially from $W^{*,^\g\Omega}_{\gamma}$.
\end{definition}

To any $\Sigma_1$ formula $\theta(v)$ in the language of Lp$^{^\gTheta\Omega}(\mathbb{R},\rm{Code}(\Omega))$ we associate formulae $\theta_k(v)$ for $k \in \omega$, such that $\theta_k$ is $\Sigma_k$, and for any $\gamma$ and any real $x$,
\begin{center}
Lp$^{^\gTheta\Omega}(\mathbb{R},\rm{Code}(\Omega))|(\gamma+1) \vDash \theta[x] \Leftrightarrow \exists k < \omega$ Lp$^{^\gTheta\Omega}(\mathbb{R},\rm{Code}(\Omega))|\gamma \vDash \theta_k[x].$
\end{center}
\begin{definition}\label{prewitness}
Suppose $\theta(v)$ is a $\Sigma_1$ formula (in the language of set theory expanded by a name for $\mathbb{R}$ and a predicate for $^\gTheta\Omega$), and $z$ is a real; then a \textbf{$\langle\theta,z\rangle$-prewitness} is an $\omega$-sound $g$-organized $\Omega$-premouse $N$ over $z$ in which there are $\delta_0 < \dots < \delta_9$, $S$, and $T$ such that $N$ satisfies the formulae expressing
\begin{enumerate}[(a)]
\item $\sf{ZFC}$,
\item $\delta_0, ..., \delta_9$ are Woodin,
\item $S$ and $T$ are trees on some $\omega\times \eta$ which are absolutely complementing in $V^{\rm{Col}(\omega,\delta_9)}$, and
\item For some $k <\omega$, $p[T]$ is the $\Sigma_{k+3}$-theory (in the language with names for each real and predicate for $^\gTheta\Omega$) of Lp$^{^\gTheta\Omega}(\mathbb{R},\rm{Code}(\Omega))|\gamma$, where $\gamma$ is least such that Lp$^{^\gTheta\Omega}(\mathbb{R},\rm{Code}(\Omega))|\gamma \vDash \theta_k[z]$.
\end{enumerate}
If $N$ is also $(\omega, \omega_1, \omega_1 + 1)$-iterable (as a $\g$-organized $\Omega$-mouse), then we call it a \textbf{$\langle\theta, z\rangle$-witness}.
\end{definition}

\begin{definition}\label{def:fine_mouse_witness_cond}
We say that the fine mouse witness condition $W_\gamma^{^g\Omega}$ holds if whenever $\theta(v)$ is a $\Sigma_1$ formula (in the language $\mathcal{L}^+$ of $\g$-organized $\Omega$-premice (cf. \cite{trang2012scales}), $z$ is a real, and Lp$^{^\gTheta\Omega}(\mathbb{R},\rm{Code}(\Omega))|\gamma\vDash \theta[z]$, then there is a $\langle\theta,z\rangle$-witness $\N$ whose $^{^g}\Omega$-iteration strategy, when restricted to countable trees on $\N$, is in Lp$^{^\gTheta\Omega}(\mathbb{R},\rm{Code}(\Omega))|\gamma$.
\end{definition}

\begin{lemma}
$W_\gamma^{*,^g\Omega}\rightarrow W_\gamma^{^g\Omega}$ for limit $\gamma$.
\end{lemma}

The proof of the above lemma is a straightforward adaptation of that of \cite[Lemma 3.5.4]{CMI}. One main point is the use of the $g$-organization: $g$-organized $\Omega$ mice behave well with respect to generic extensions in the sense that if $\P$ is a $g$-organized $\Omega$ mouse and $h$ is set generic over $\P$ then $\P[h]$ can be rearranged to a $g$-organized $\Omega$ mouse over $h$.

The induction is guided by the pattern of scales in Lp$^{^\gTheta\Omega}(\mathbb{R},\rm{Code}(\Omega))$ as analyzed in \cite{trang2012scales}. To show $\sf{AD}^+ + \sf{SMC}$ holds in $\rm{sLp}^{^\gTheta\Omega}(\mathbb{R},\rm{Code}(\Omega))$, we show $\rm{sLp}^{^\gTheta\Omega}(\mathbb{R},\rm{Code}(\Omega))\vDash \forall \alpha (\alpha \textrm{ is critical} \rightarrow W_\alpha^{*,^g\Omega}$). Our plan is to show $W^{*,^\g\Omega}_{\alpha+1}$ assuming $W^{*,^\g\Omega}_{\alpha}$ for $\alpha$ critical. Lemma \ref{lemma-not-square-implies-M1Fsharp} and the subsequent corollary provide the base case for our induction. For $\alpha > 0$, we have three cases: 
\begin{enumerate}
\item $\alpha$ is a successor of a critical ordinal or $\alpha$ is a limit of critical ordinals and cof$(\alpha)$ = $\omega$; 
\item $\alpha$ is inadmissible, limit of critical ordinals, cof$(\alpha) > \omega$
\item $\alpha$ ends a weak gap or successor of an ordinal that ends a strong gap. Say the gap is $[\gamma,\alpha^*]$, where $\alpha^*= \alpha$ if the gap is weak and $\alpha^*+1=\alpha$ if the gap is strong. Furthermore, $\rm{sLp}^{^\gTheta\Omega}(\mathbb{R},\rm{Code}(\Omega))|\alpha \vDash \sf{MC}$$(\Sigma) + \sf{AD}^+ + \Theta = \theta_\Sigma$.
\end{enumerate}
\indent We deal with the easy case (case 1) first. In this case, let $\Gamma = \Sigma_1(\rm{sLp}^{^\gTheta\Omega}(\mathbb{R},\rm{Code}(\Omega))|\alpha)$. Then $C_{\Gamma} = \bigcup_{n<\omega}C_{\Gamma_n}$ for some increasing sequence of scaled pointclasses $\langle \Gamma_n \ | \ n < \omega\rangle$. By $W^{*,^g\Omega}_\alpha$, for each $n$, we have $\Sigma$-cmi operators $\langle J^n_m \ | \ m < \omega\rangle$ that collectively witness $\sf{AD}$$^{\Gamma_n}$. Say each $J^n_m$ defined on a cone above some fixed $a\in \HC$. The desired mouse operator $J_0$ is defined as follows: For each transitive and self-wellordered $A \in \HC$ coding $a$, $J_0(A)$ is the shortest initial segment $\M \triangleleft \rm{Lp}^{^\g\Omega}(A)$ such that $\M \vDash \sf{ZFC}^-$ and $\M$ is closed under $J^m_n$ for all $m,n$. $J_0$ is total and trivially relativizes well and determines itself on generic extensions because the $J^m_n$'s have these properties. We then use Lemma \ref{lemma-not-square-implies-M1Fsharp} to get that $J_1 = \M_1^{\sharp,J_0}$ is defined on the cone above $a$ by arguments in the previous section. Inductively, we get that $J_{n+1}= \M_1^{\sharp,J_n}$ is defined on the cone above $a$ for all $n$ and one easily gets that these operators are $\Sigma$-cmi operators. By Lemma 4.1.3 of \cite{CMI}, this implies $W^{*,^\g\Omega}_{\alpha+1}$.
\\
\indent Now we're on to the case where $\alpha$ is inadmissible and cof$(\alpha) > \omega$. Let $\phi(v_0,v_1)$ be a $\Sigma_1$ formula and $x\in \mathbb{R}$ be such that 
\begin{equation*}
\forall y\in \mathbb{R} \ \exists \beta < \alpha \ \rm{sLp}^{^\gTheta\Omega}(\mathbb{R},\rm{Code}(\Omega))|\beta\vDash \phi[x,y],
\end{equation*}
and letting $\beta(x,y)$ be the least such $\beta$,
\begin{equation*}
\alpha = \textrm{sup}\{\beta(x,y) \ | \ y \in \mathbb{R}\}.
\end{equation*}
We first define $J_0$ on transitive and self-wellordered $A \in \HC$ coding $x$. For $n < \omega$, let
\begin{equation*}
\phi^*(n) \equiv \exists \gamma (\textrm{Lp}^{^\gTheta\Omega}(\mathbb{R},\rm{Code}(\Omega))|\gamma \vDash \forall i\in \omega(i>0 \Rightarrow \phi((v)_0,(v)_1) \wedge (\gamma+\omega n) \textrm{ exists})).
\end{equation*}
For such an $A$ as above, let $\M$ be an $A$-premouse and $G$ be a Col$(\omega,A)$-generic over $\M$, then $\M[G]$ can be regarded as a $\g$-organized $\Omega$-mouse over $z(G,A)$ where $z(G,A)$ is a real coding $G,A$ and is obtained from $G,A$ in some simple fashion.\footnote{This is one of the main reasons that we consider $^g\Omega$-mice; this is so that generic extensions of $^g\Omega$-mice can be rearranged to $^g\Omega$-mice.} Also, let $\sigma_A$ be a term defined uniformly (in $\M$) from $A, x$ such that
\begin{equation*}
(\sigma_A^G)_0 = x
\end{equation*}
and
\begin{equation*}
\{(\sigma_A^G)_i \ | \ i > 0\} = \{ \rho^G \ | \ \rho \in L_1(A) \wedge \rho^G \in \mathbb{R}\}.
\end{equation*}
Let $\varphi$ be a sentence in the language of $A$-premice such that for any $A$-premouse $\M$, $\M \vDash \varphi$ iff whenever $G$ is $\M$-generic for $Col(\omega,A)$, then for any $n$ there is a $\gamma < o(\M)$ such that
\begin{equation*}
\M[z(G,A)]|\gamma \textrm{ is a } \langle\phi^*_n,\rho_A^G\rangle\textrm{-prewitness}.
\end{equation*}

Then $J_0(A)$ is the shortest initial segment of $\rm{Lp}^{^\g\Omega}(A)$ which satisfies $\varphi$, if it exists, and is undefined otherwise. Using the fact that $W_\alpha^{^g\Omega}$ holds, we get that  $J_0(A)$ exists for all $A\in \HC$ coding $x$ because $\alpha$ has uncountable cofinality and there are only countably many $\langle\phi^*_n,\rho_A^G\rangle$.  Also we can then define $J_n$ as before. It's easy to show again that the $J_n$'s relativize well and determine themselves on generic extensions; and so they are $\Sigma$-cmi operators. This implies $W_{\alpha+1}^{*,^g\Omega}$.
\\
\indent Lastly, we consider the gap case. Using the notations as in (3) above, let $\Gamma = \Sigma_1^{\rm{sLp}^{^\gTheta\Omega}(\mathbb{R},\rm{Code}(\Omega))|\gamma}$ and $T$ be a tree projecting to a universal $\Gamma$ set. If $[\gamma,\alpha^*]$ is a weak gap, by the scales analysis at the end of a weak gap from \cite{Scalesendgap} and \cite{trang2012scales}, we can construct a self-justifying system (sjs) $\mathcal{A}$ Wadge-cofinal in $\powerset(\mathbb{R})\cap \rm{sLp}$$^{^\gTheta\Omega}(\mathbb{R},\rm{Code}(\Omega))|\alpha^*$.\footnote{This means $\mathcal{A}$ is a countable collection containing a universal $\Sigma_1^{\rm{sLp}^{^\gTheta\Omega}(\mathbb{R})|\gamma}$ set, closed under complements and whenever $A\in\mathcal{A}$, then there is a scale whose individual norms are coded by sets in $\mathcal{A}$.} If $[\gamma,\alpha^*]$ is a strong gap, by the Kechris-Woodin theorem, $\sf{AD}^+$ holds in $\rm{sLp}$$^{^\gTheta\Omega}(\mathbb{R},\rm{Code}(\Omega))|\alpha$, and again by results of \cite{Scalesendgap}, \cite{trang2012scales}, and \cite{wilson2015envelope}, we also get a self-justifying-system $\mathcal{A}$ Wadge-cofinal in $\rm{sLp}$$^{^\gTheta\Omega}(\mathbb{R},\rm{Code}(\Omega))|\alpha\cap \powerset(\mathbb{R})$.  From $\mathcal{A}$ and arguments in \cite[Section 5]{CMI}, there is a pair $(\Q,\Lambda)$ such that $\Q$ is $\Gamma$-suitable and $\Lambda$ is the $(\omega_1,\omega_1)$-strategy for $\Q$ guided by $\mathcal{A}$. Let $J_0 = \Lambda$. 
\begin{claim}\label{FDetGenExts}
$J_0$ determines itself on generic extensions.
\end{claim}
\begin{proof}
Let $\kappa=_{\rm{def}}\omega_1^V \subset N$ be a transitive structure of $\sf{ZFC}^-$ such that $N$ is closed under $J_0$ (and hence under $\Lambda$); we may assume also that $T\in N$. We simply describe a procedure that determines $\Lambda$ on generic extensions of $N$; the reader may gladly verify that this is enough to prove the claim. Let $g\in V$ be generic over $N$ and let $\T$ be a tree according to $\Lambda$ of limit length in $N[g]$ (the argument for stacks is similar). If $\T$ is short, using $T$, we can find the $\Q$-structure $\Q(\T)$ for $\T$ and this in turns determines the branch $b=\Lambda(\T)\in N[g]$. Suppose $\T$ is maximal. By boolean comparison (cf. \cite[Section 5.4]{CMI}), we can find a tree $\U\in N$ according to $\Lambda$ of length $< \kappa$ \footnote{This uses that $\kappa$ is inaccessible in $N$.} such that 
\begin{enumerate}[(i)]
\item $\U$ is nondropping with last model $\M^\U$ and branch embedding $\pi^\U$;
\item $\Lambda(\T)=b$ is the unique branch in $N[g]$ with last model $\M^\T$ and branch embedding $\pi^\T$ such that there is an embedding $\tau:\M^\T\rightarrow \M^\U$ $\pi^\U =\tau\circ \pi^\T$. 
\end{enumerate}
\end{proof}
Furthermore, $J_0$ is suitable (we can construct $\M_1^{J_0,\sharp}$ by arguments in the previous section) and $\M_1^{J_0,\sharp}$ generically interprets $J_0$ by \cite[Lemma 4.8]{trang2012scales}. Note that $J_0$ and $\cA$ are projectively equivalent in any reasonable coding. We can use Lemma \ref{lemma-not-square-implies-M1Fsharp} to show $W^{*,^g\Omega}_{\alpha+1}$ by constructing a sequence of operators $(J_n : n<\omega)$, where $J_{n+1} = \M_1^{J_n,\sharp}$ for all $n$. \footnote{These operators, again, can be shown to be $\Sigma$-cmi operators. Here and elsewhere, we suppress the formula $\varphi_{\rm{min}}$ defined in \cite[Definition 3.2]{trang2012scales} from the definition of $J_1= \M_1^{J_0,\sharp}$; to be entirely correct, according to \cite{trang2012scales}, $J_1$ should be $\M_1^{(J_0,\varphi_{\rm{min}}),\sharp}$.}

It now follows easily that we can strengthen the conclusion of Corollary \ref{corollary:not-square-implies-PD} to obtain the following result.

\begin{corollary}\label{corollary:not-square-implies-AD-LR}
 Assume $\mathsf{ZF} + \mathsf{DC}_{\powerset(\omega_1)} + {}$``$\omega_1$ is $\mathbb{R}$-strongly compact and $\neg\square_{\omega_1}$.'' Then $\mathsf{AD}$ holds in $L(\mathbb{R})$.
\end{corollary}

This corollary completes the proof of Theorem \ref{theorem:AD-equiconsistency}.  It also forms a significant first step in the proofs of Theorems \ref{theorem:ADR-equiconsistency} and \ref{theorem:ADR-DC-equiconsistency}.

\section{A MODEL OF $\mathsf{AD}^+ + \Theta > \theta_\Sigma$}\label{section:Theta-gt-theta-Sigma}

Suppose $(\P,\Sigma)$ is a $\G$-$\Omega^*$-suitable pair for some nice operator $\G$ such that $\Sigma$ has branch condensation and is $\Omega^*$-fullness preserving.
As a special case we also allow $(\P,\Sigma) = (\emptyset, \emptyset)$.
In the previous section we showed (under our strong hypotheses plus a smallness assumption)
that there is a maximal model of $\mathsf{AD}^+ + V = L(\powerset(\mathbb{R})) + \Theta = \theta_\Sigma$ containing all reals and ordinals.
This model has the form $L(\Omega_\Sigma,\mathbb{R})$ where
$L(\Omega_\Sigma,\mathbb{R}) \cap \powerset(\mathbb{R}) = \Omega_\Sigma$.
In this section, we will go just beyond this model to obtain a model of $\mathsf{AD}^+ + \Theta > \theta_\Sigma$ containing all reals and ordinals.

Define the pointclass 
\[\Gamma = (\Sigma^2_1(\Code(\Sigma))^{\Omega^*}.\]
Note that
we have $\Gamma = (\Sigma^2_1(\Code(\Sigma))^{\Omega_\Sigma}$; this is because if a set of reals $A \in \Omega^*$ witnesses a $\Sigma^2_1(\Code(\Sigma))$ fact about a real $x$, then there is a set of reals in 
$\Delta^2_1(\Code(\Sigma), x)^{L(A,\mathbb{R})}$ witnessing the same fact about $x$ by Woodin's $\Delta^2_1$ basis theorem relativized to $x$ and $\Code(\Sigma)$ and applied in the model $L(A,\mathbb{R})$, and such a set of reals can be shown to be in $\Omega_\Sigma$.

Recall from Section \ref{CMIInsideLp} that (under our smallness assumption) the maximal model $L(\Omega_\Sigma,\mathbb{R})$ of $\mathsf{AD}^+ + \Theta = \theta_\Sigma$ is, up to its $\Theta$, a hybrid mouse over $\mathbb{R}$ of the form $\rm{sLp}^{^\gTheta\Omega}(\mathbb{R},\rm{Code}(\Omega))$
where we have defined the operator $\Omega=\Sigma$. We remind the reader that Code$(\Omega)$ is self-scaled.

In particular we have
\[ \Omega_\Sigma = \powerset(\mathbb{R}) \cap \rm{sLp}^{^\gTheta\Omega}(\mathbb{R},\rm{Code}(\Omega)),\]
so we can reformulate our pointclass as
\[ \Gamma = (\Sigma^2_1)^{\rm{sLp}^{^\gTheta\Omega}(\mathbb{R},\rm{Code}(\Omega))} = (\Sigma^2_1)^{\rm{sLp}^{^\gTheta\Omega}(\mathbb{R},\rm{Code}(\Omega))|\alpha} \]
where $\alpha = ({\bf \delta}^2_1)^{\rm{sLp}^{^\gTheta\Omega}(\mathbb{R},\rm{Code}(\Omega))}$ is the ordinal beginning the last gap of $\rm{sLp}^{^\gTheta\Omega}(\mathbb{R},\rm{Code}(\Omega))$.
(Recall that by $\Sigma^2_1$ we mean to include $\Omega$, or equivalently $\Sigma$, as a parameter. By self-iterability it makes no difference whether we also include the extender sequence as a parameter.)

Like the pointclass considered in the ``gap in scales'' case of the core model induction in Section \ref{CMIInsideLp}, the pointclass $\Gamma$
is an inductive-like pointclass with the scale property.
Our next task is to find the next scaled pointclass, or (what is roughly equivalent) to build a scale on a complete $\check{\Gamma}$ set.
Unlike in Section \ref{CMIInsideLp}, this next scaled pointclass cannot be found within $\rm{sLp}^{^\gTheta\Omega}(\mathbb{R},\rm{Code}(\Omega))$.  The reason is that 
the complete $\check{\Gamma}$ set $\big\{(x,y) \in \mathbb{R} \times \mathbb{R} : y \notin \OD_x^{\rm{sLp}^{^\gTheta\Omega}(\mathbb{R},\rm{Code}(\Omega))}\big\}$ cannot have any uniformization in $\rm{sLp}^{^\gTheta\Omega}(\mathbb{R},\rm{Code}(\Omega))$, and therefore cannot have any scale in $\rm{sLp}^{^\gTheta\Omega}(\mathbb{R},\rm{Code}(\Omega))$, by a standard argument.

We will use our strong hypotheses (as in Theorems \ref{theorem:ADR-equiconsistency} and \ref{theorem:ADR-DC-equiconsistency})
to build a scale on a complete $\check{\Gamma}$ set.
Each prewellordering of this scale will be in $L(\Omega_\Sigma,\mathbb{R})$, or equivalently in $\rm{sLp}^{^\gTheta\Omega}(\mathbb{R},\rm{Code}(\Omega))$, although the sequence of prewellorderings cannot be, as we just saw.

More directly, what we will show is that the prewellorderings are in a pointclass $\Env(\Gamma)$, the envelope of $\Gamma$.
This notion was used by Martin to identify the next scaled pointclass after an inductive-like scaled pointclass in the $\mathsf{AD}$ context; see Jackson \cite{Jackson}.
We will need its adaptation to the partial determinacy context as defined in the second author's thesis \cite{wilson2012contributions} (see also the subsequent paper \cite{wilson2015envelope}.)

It turns out that $\Env(\Gamma) \subset L(\Omega_\Sigma,\mathbb{R})$, and in fact $\Env(\Gamma)$ consists exactly of the sets of reals that are ordinal-definable from $\Sigma$ in the model $L(\Omega_\Sigma,\mathbb{R})$, but we will not be able to see this until later.
For now we must use the following ``local'' definition of the envelope in terms of the ambiguous pointclass $\Delta_\Gamma = \Gamma \cap \check{\Gamma}$ and in terms of the notion of ``$\Delta_\Gamma$ in an ordinal parameter.''
This notion can be defined in general, but here we can take the following characterization as a definition: a set of reals is \emph{$\Delta_\Gamma$ in an ordinal parameter} if and only if $\Delta_1$-definable over $\rm{sLp}^{^\gTheta\Omega}(\mathbb{R},\rm{Code}(\Omega))|\alpha$ from ordinals (and $\Omega$, or equivalently $\Sigma$.)

%
%

\begin{definition}
The envelope of $\Gamma$, denoted by $\Env(\Gamma)$, is the pointclass consisting of all pointsets $A$ such that, for every countable $\sigma \subset \mathbb{R}$, there is a pointset $A'$ that is $\Delta_\Gamma$ in an ordinal parameter and satisfies $A \cap \sigma = A' \cap \sigma$.

The boldface pointclass $\bfEnv(\Gamma)$ is defined similarly but allowing a real parameter. That is, $A \in \bfEnv(\Gamma)$ if there is a real $x$ such that for every countable $\sigma \subset \mathbb{R}$ there is a pointset $A'$ that is $\Delta_\Gamma(x)$ in an ordinal parameter and satisfies $A \cap \sigma = A' \cap \sigma$.
\end{definition}

The following fact about envelopes is crucial for our argument.
It is essentially proved in the thesis
\cite{wilson2012contributions} (which deals with generic large cardinal properties of $\omega_1$ in $\mathsf{ZFC}$ rather than with large cardinal properties of $\omega_1$, but the argument carries over to the present context.)
An easier version with ``scale'' replaced by ``semiscale'' is proved in the paper \cite{wilson2015envelope}, and a special case of the scale construction appears in another paper \cite{wilson2015scales}.

\begin{lemma}[Wilson]\label{lemma:scale-from-strong-compactness}
 Assume $\mathsf{ZF} + \mathsf{DC}$. Let $\Gamma$ be an inductive-like pointclass with the scale property.
 Suppose that $\omega_1$ is $\Env(\Gamma)$-strongly compact.
 Then there is a scale on a  universal $\check{\Gamma}$ set, each of whose prewellorderings is in $\Env(\Gamma)$.
\end{lemma}

Another important fact about envelopes is that if $\mathsf{ZF} + \mathsf{DC}_\mathbb{R}$ holds and the boldface ambiguous part $\bfDelta_\Gamma$ of the pointclass $\Gamma$ is determined, as it is here, then $\bfEnv(\Gamma)$ is determined and projectively closed (Wilson \cite{wilson2012contributions, wilson2015envelope}; based on work of Kechris, Woodin, and Martin.)
Therefore Wadge's lemma applies to it, as one can easily verify that the relevant games are determined.
 Moreover, the Wadge preordering\footnote{We are abusing notation here; really it is a preordering of pairs $\{B,\neg B\}$ where $B \in \bfEnv(\Gamma)$.} of $\bfEnv(\Gamma)$ is a prewellordering:
 otherwise by $\mathsf{DC}_\mathbb{R}$ we could choose a sequence of pointsets in $\bfEnv(\Gamma)$ that was strictly decreasing in the Wadge ordering, but then by the proof of the Martin--Monk theorem we get a contradiction.
 (Again one can easily verify that the relevant games are determined.)

Note that the prewellorderings of a scale as in Lemma \ref{lemma:scale-from-strong-compactness} must be Wadge-cofinal in $\bfEnv(\Gamma)$; otherwise the sequence of prewellorderings itself would be coded by a set of reals in $\bfEnv(\Gamma)$, which is impossible as mentioned above.
From such a scale, it then follows by a general argument (see Jackson \cite{Jackson} and the straightforward adaptation \cite[Section 4.3]{wilson2012contributions} to the partial determinacy context) that we can obtain a self-justifying system contained in $\bfEnv(\Gamma)$.\footnote{We don't know if it is possible to obtain a self-justifying system contained in the lightface envelope, but this will not matter for our application.}

\begin{lemma}\label{lemma:sjs-from-strong-compactness}
 Assume $\mathsf{ZF} + \mathsf{DC}$. Let $\Gamma$ be an inductive-like pointclass with the scale property such that $\bfDelta_\Gamma$ is determined.
 Suppose that $\omega_1$ is $\Env(\Gamma)$-strongly compact.
 Then there is a self-justifying system $\mathcal{A} \subset \bfEnv(\Gamma)$ containing a universal $\Gamma$ set.
\end{lemma}

We will use this lemma together with the hypotheses of Theorems \ref{theorem:ADR-equiconsistency} or \ref{theorem:ADR-DC-equiconsistency}, to obtain 
a self-justifying system $\mathcal{A} \subset \bfEnv(\Gamma)$ containing a universal $\Gamma$ set.
We begin with the observation that the length of the Wadge prewellordering of $\bfEnv(\Gamma)$ is at most $\Theta$ by the usual argument: the initial segment corresponding to a set $B \in \bfEnv(\Gamma)$ is the image of $\mathbb{R}$
 under the function $y \mapsto g_y^{-1}[B]$, where $g_y$ denotes the continuous function coded by the real $y$.
 Moreover, the lightface envelope $\Env(\Gamma)$ admits a wellordering (essentially an ultrapower of the canonical wellordering of the $\Delta_\Gamma$-in-an-ordinal sets by Martin's cone measure, which measures the relevant sets by $\Env(\Gamma)$-determinacy.)

\begin{lemma}
 Let $\Gamma$ be an inductive-like pointclass with the scale property such that $\bfDelta_\Gamma$ is determined.
 Assume $\mathsf{ZF} + \mathsf{DC} + {}$``$\omega_1$ is $\Theta$-strongly compact.'' Then there is a self-justifying system $\mathcal{A} \subset \bfEnv(\Gamma)$ containing a universal $\Gamma$ set.
\end{lemma}
\begin{proof}
 Consider the restriction of the Wadge prewellordering of $\bfEnv(\Gamma)$ to the lightface envelope $\Env(\Gamma)$.
 We can refine this prewellordering to a wellordering
 by taking its lexicographical product with a wellordering of $\Env(\Gamma)$, which exists, as mentioned above.
 This refinement has the property that
 its length is at most $\Theta$, because its initial segment below any set $A \in \Env(\Gamma)$ is contained in the Wadge-initial segment $\{B \in \Env(\Gamma) : B \le_{\text{W}} A\}$.
 (It's not clear whether the original wellordering of $\Env(\Gamma)$ described above has this property.)
 Therefore our hypothesis implies that $\omega_1$ is $\Env(\Gamma)$-strongly compact, and the desired conclusion follows by Lemma \ref{lemma:sjs-from-strong-compactness}.
\end{proof}

\begin{lemma}
 Let $\Gamma$ be an inductive-like pointclass with the scale property such that $\bfDelta_\Gamma$ is determined.
 Assume $\mathsf{ZF} + \mathsf{DC}_\mathbb{R} + {}$``$\omega_1$ is $\mathbb{R}$-strongly compact and $\Theta$ is singular.'' Then there is a self-justifying system $\mathcal{A} \subset \bfEnv(\Gamma)$ containing a universal $\Gamma$ set.
\end{lemma}
\begin{proof}
 Let $\lambda$ denote the length of the Wadge prewellordering of $\bfEnv(\Gamma)$ and fix a cofinal function $\mathbb{R} \to \lambda$, say $x \mapsto \beta_x$. (The case $\lambda = \Theta$ is where we use the assumption that $\Theta$ is singular, although it turns out that this case cannot occur when $\Theta$ is singular.)
 Note that the lightface pointclass $\Env(\Gamma)$ is Wadge-cofinal in the boldface pointclass $\bfEnv(\Gamma)$ because every subset of $\mathbb{R}$ in $\bfEnv(\Gamma)$ is a section of some subset of $\mathbb{R} \times \mathbb{R}$ in $\Env(\Gamma)$.

 Fix a wellordering $<$ of the lightface envelope $\Env(\Gamma)$.
 Then to each real $x$ we can assign the $<$-least set $B_x \in \Env(\Gamma)$ whose rank in the Wadge prewellordering of $\bfEnv(\Gamma)$ is at least $\beta_x$.
 Then the family of sets $\{B_x : x \in \mathbb{R}\}$ is cofinal in the Wadge prewellordering of $\bfEnv(\Gamma)$, and can we obtain a surjection from $\mathbb{R} \times \mathbb{R}$ onto $\bfEnv(\Gamma)$ by
 $(x,y) \mapsto g_y^{-1}[B_x]$, where $g_y$ is the continuous function coded by the real $y$.

 Therefore there is a surjection from $\mathbb{R}$ onto $\bfEnv(\Gamma)$,
 and by our hypothesis that $\omega_1$ is $\mathbb{R}$-strongly compact, it follows that $\omega_1$ is $\bfEnv(\Gamma)$-strongly compact.  In particular it is $\Env(\Gamma)$-strongly compact, which is all we need.
 We could now apply Lemma \ref{lemma:sjs-from-strong-compactness} to obtain the desired conclusion, except for the problem that we only have $\mathsf{DC}_\mathbb{R}$ in place of $\mathsf{DC}$.  This problem can be solved by passing to an inner model.

 Take a fine, countably complete measure $\mu$ on $\powerset_{\omega_1}(\Env(\Gamma))$ and consider the
 model $L(X)[\mu]$
 where $X = \Env(\Gamma)^\omega \cup \mathbb{R}$.
 In $V$ we have $\mathsf{DC}_\mathbb{R}$ and we have a surjection from $\mathbb{R}$ to  $X$, so we have $\mathsf{DC}_X$.
 Because an $\omega$-sequence of elements of $X$ can be coded by a single element of $X$, we have $\mathsf{DC}_X$ in  $L(X)[\mu]$ as well.
 In $L(X)[\mu]$ every set is a surjective image of $X \times \xi$ for some ordinal $\xi$, so $\mathsf{DC}$ follows from $\mathsf{DC}_X$ by a standard argument.
 Then we can apply Lemma \ref{lemma:sjs-from-strong-compactness} in $L(X)[\mu]$ and note that the conclusion is upward absolute to $V$.
\end{proof}

Now that we have obtained a self-justifying system $\mathcal{A} = (A_i : i<\omega)$ sealing the envelope of $\Gamma$,
we may proceed as in the ``gap in scales'' case of
Section \ref{CMIInsideLp}
to get a pair $(\mathcal{Q}, \Lambda)$ such that $\mathcal{Q}$ is an
$\Gamma$-suitable ($\g$-organized) $\Omega$-premouse and $\Lambda$ is the $(\omega_1,\omega_1)$-iteration strategy for $\mathcal{Q}$ guided by $\mathcal{A}$.
A slight difference from Section \ref{CMIInsideLp} is caused by the fact that, at this stage in the argument, we do not know how to rule out the possibility that the pointclass $\bfEnv(\Gamma)$ properly contains the pointclass $\Omega_\Sigma = \powerset(\mathbb{R})\cap \rm{sLp}^{^\gTheta\Omega}(\mathbb{R},\rm{Code}(\Omega))$.

However, this difference does not create any problem because
the important thing is that every set $A \in \bfEnv(\Gamma)$
(and in particular every set $A_i$ in our self-justifying system
$\mathcal{A}$) has the property that,
for a cone of $b \in \HC$,
the hybrid lower part mouse $\rm{Lp}^{^\g\Omega,\Gamma}(b)$
has a $\Coll(\omega,b)$-term for a set of reals that locally captures $A$.
(If $A$ is in the lightface envelope then the base of the cone is $\emptyset$ and this holds for all $b \in \HC$.)
For a proof, see Wilson \cite[Section 4.2]{wilson2012contributions}.
This local term-capturing property is sufficient to make sense of the notion of $A$-iterability, to prove the existence of $A$-iterable premice, and to get an iteration strategy $\Lambda$ guided by the self-justifying system $\mathcal{A}$.

Defining the $\Sigma$-CMI operator $\mathcal{F} = \Lambda$,
we can then use Lemma \ref{lemma-not-square-implies-M1Fsharp} to construct a sequence of $\Sigma$-CMI operators $(J_n : n<\omega)$, where $J_0 = \mathcal{F}$ and $J_{n+1} = \M_1^{\sharp, J_n}$ for all $n>0$.
Because $\mathcal{A}$ and $\mathcal{F}$ are projectively equivalent (in any reasonable coding)
this shows the existence of a determined projective-like hierarchy just beyond $\bfEnv(\Gamma)$, and therefore beyond the maximal model of $\mathsf{AD}^+ + \Theta = \theta_\Sigma$.

To continue further and get a model of $\mathsf{ZF} + \mathsf{AD}^+ + \Theta > \theta_\Sigma$, we proceed along the lines of
Section \ref{CMIInsideLp}.
The difference is that now the operator $\mathcal{F}$ is here to stay: we must
consider $\mathcal{F}$-hybrid mice from this point on,
and never return to considering $\Omega$-hybrid mice
because they cannot give us anything new.

Our model of $\mathsf{AD}^+ + \Theta > \theta_\Sigma$
will be obtained as the maximal model of 
$\mathsf{AD}^+ + \Theta = \theta_\Lambda$ (and $\theta_\Sigma$ will be the penultimate member of its Solovay sequence.)
The existence of this maximal model is established by the results of
Section \ref{CMIInsideLp}
with the suitable pair $(\mathcal{Q},\Lambda)$ and its associated operator $\mathcal{F}$ in place of the hod pair (or suitable pair, or empty pair) $(\mathcal{P},\Sigma)$ and its associated operator $\Omega$.
(For this reason it is important that we allowed suitable pairs as well as hod pairs and empty pairs in Sections \ref{OneStep} and \ref{CMIInsideLp}.)

To obtain the maximal model of $\mathsf{AD}^+ + \Theta = \theta_\Lambda$,
it remains only to show that $\Lambda$ can be extended to a $\Theta+1$-iteration strategy with branch condensation. (In fact, we will show that it can be extended to a $\Theta^+$-iteration strategy with branch condensation.) As remarked in Section \ref{CMIInsideLp}, this strategy extension is necessary to define the model $\rm{sLp}^{^\gTheta\F}(\mathbb{R},\rm{Code}(\F))$ via $\g$-organization, which in turn is necessary to analyze the pattern of scales in this model.

Note that because the iteration strategy $\Lambda$ is guided by a self-justifying system, it has branch condensation and hull condensation
and the set of reals coding it is Suslin.
Accordingly, we can use the following lemma to extend $\Lambda$.
Our argument is based on Schindler and Steel \cite[Lemmas 2.1.11 and 2.1.12]{CMI},
but some adaptations are necessary in the absence of $\mathsf{AC}$.
A similar argument is also found in Steel \cite{PFA}.

Before proving the lemma (which will take the remainder of this section)
let us note that the hypothesis that every uncountable regular cardinal $\le \Theta$ is threadable
follows from the hypotheses of Theorems \ref{theorem:ADR-equiconsistency} and \ref{theorem:ADR-DC-equiconsistency}.
(In particular, it follows from the hypothesis $\mathsf{ZF} + \mathsf{DC} + {}$``$\omega_1$ is $\Theta$-strongly compact'' and also from the hypothesis $\mathsf{ZF} + \mathsf{DC}_{\mathbb{R}} + {}$``$\omega_1$ is $\mathbb{R}$-strongly compact and $\Theta$ is singular.'')
Note also that the conclusion that the extension of $\Lambda$ has hull condensation, together with the fact that the original $\omega_1$-iteration strategy $\Lambda$ has branch condensation, implies that the
extension strategy also has branch condensation by an easy Skolem hull argument. (We can take the Skolem hull in an inner model of $\mathsf{ZFC}$, so that no choice is required.)

\begin{lemma}\label{lemma-extend-hull-condensation}
 Assume that $\mathsf{ZF}$ holds and let $\Lambda$ be an $\omega_1$-iteration strategy with hull condensation for a premouse\footnote{By a premouse here we mean an $\mathcal{F}$-premouse where $\mathcal{F}$ is an operator that condenses finely (such as the core model induction operators that we consider in this paper.)  Alternatively we could use coarse mice here, because we will only need the extended strategy for genericity iterations.} $\mathcal{Q}$.  Assume that $\operatorname{Code}(\Lambda)$ is Suslin.  Let $\eta$ be an uncountable cardinal and assume that every uncountable regular cardinal $\le \eta$ is threadable.
 Then $\Lambda$ has a (necessarily unique) extension to an $\eta^+$-iteration strategy with hull condensation.
\end{lemma}
\begin{proof}
 Let $\mathcal{T}$ be a putative iteration tree on $\mathcal{Q}$ of length less than $\eta^+$
 and such that every countable hull of $\mathcal{T}$ is by $\Lambda$.
 (A putative iteration tree is like an iteration tree except that its last model, if it has one, is allowed to be illfounded.)
 What we want to show is that if $\mathcal{T}$ has a last model, then this last model is wellfounded, and if $\mathcal{T}$ has limit length, then it has a unique cofinal wellfounded branch $b$ such that 
 every countable hull $\bar{\mathcal{T}}^\frown \bar{b}$ of $\mathcal{T}^\frown b$ is also by $\Lambda$ (in which case our extension of $\Lambda$ can and must choose this branch.)
 
 In the case that $\mathcal{T}$ has a last model, it is easy to see that the last model must be wellfounded; otherwise by taking a Skolem hull (of $L_{\eta^+}[\mathcal{Q},\mathcal{T}]$, say, so that no choice is required) we may obtain a countable hull of $\mathcal{T}$ whose last model is illfounded, but the last model of the hull must be wellfounded because the hull is by the iteration strategy $\Lambda$.
 
 Now suppose that $\mathcal{T}$ has limit length.
 This case will require a bit more work.  First we note that it suffices to find some cofinal branch $b$ of $\mathcal{T}$ such that every countable hull $\bar{\mathcal{T}}^\frown \bar{b}$ of $\mathcal{T}^\frown b$ is by $\Lambda$; then a Skolem hull argument shows that there can be at most one such branch and that any such branch is wellfounded.
  Let $q$ be a real coding the premouse $\mathcal{Q}$. We consider two subcases.

 \begin{enumerate}
  \item\label{item-T-uncountable-cofinality} $\lh(\mathcal{T})$ has uncountable cofinality.
\end{enumerate}

 In this subcase,
 we use the general fact about iteration trees that
 the sequence of branches  $[0,\alpha)_{\mathcal{T}}$ for limit ordinals $\alpha < \lh(\mathcal{T})$ is a coherent sequence of clubs.
 Here $\lh(\mathcal{T})$ is threadable (equivalently, has threadable cofinality,)
 so the tree $\mathcal{T}$ has a unique cofinal branch $b$ obtained by threading this coherent sequence.
 Let $\bar{\mathcal{T}}^\frown \bar{b}$ be a countable hull of $\mathcal{T}^\frown b$.
 We want to show that $\bar{\mathcal{T}}^\frown \bar{b}$ is by $\Lambda$.
 
 Let $x$ be a real coding $\bar{\mathcal{T}}^\frown \bar{b}$.
 The model $N = L[q,\mathcal{T},b,\Lambda,x]$\footnote{We are abusing notation here.  For example, instead of $\Lambda$ itself as a predicate we mean $\{(\mathcal{U},\xi) : \xi \in \Lambda(\mathcal{U})\}$.}
 satisfies $\mathsf{AC}$ and therefore $\square_\omega$, whereas $V$ satisfies ``$\omega_1$ is threadable'' and therefore $\neg \square_\omega$, so $\omega_1^N < \omega_1$.
 Note that the model $N$ sees that $\bar{\mathcal{T}}^\frown \bar{b}$ is a hull of $\mathcal{T}^\frown b$ by the absoluteness of wellfoundedness for the tree of attempts to build a map
 $\lh(\bar{\mathcal{T}}) \to \lh(\mathcal{T})$ witnessing this (or we could just put such a map into the model.)
 The model $N$ also sees, of course, that $\lh(\mathcal{T})$ has uncountable cofinality.
 
 Working in $N$,
 by a Skolem hull argument we can take a hull $\mathcal{T}^{* \frown} b^*$ of $\mathcal{T}^\frown b$ such that
 $\lh(\mathcal{T}^*)$ has cardinalilty and cofinality $\omega_1$ 
 and $\bar{\mathcal{T}}^\frown \bar{b}$ is a hull of $\mathcal{T}^{* \frown} b^*$.
 Because the tree $\mathcal{T}^*$ is countable in $V$ the branch $\Lambda(\mathcal{T}^*)$ is defined,
 and the model $N$ can see it.
 In $N$ the tree $\mathcal{T}^*$ can have at most one cofinal branch because its length has uncountable cofinality, so  $\Lambda(\mathcal{T}^*) = b^*$.
 Therefore the hull $\mathcal{T}^{* \frown} b^*$ is by $\Lambda$, and by hull condensation \emph{its} hull $\bar{\mathcal{T}}^\frown \bar{b}$ is also by $\Lambda$, as desired.
 
 \begin{enumerate}
 \setcounter{enumi}{1}
   \item $\lh(\mathcal{T})$ has countable cofinality.
 \end{enumerate}
 
 In this subcase, 
 we define an elementary substructure $X \prec L_{\eta^+}[\mathcal{Q},\mathcal{T}]$ in $V$ to be \emph{appropriate}
 if $\mathcal{Q} \cup \{\mathcal{Q},\mathcal{T}\} \subset X$, $X$ is countable, and $X \cap \lh(\mathcal{T})$ is cofinal in $\lh(\mathcal{T})$.
 For an appropriate elementary substructure $X \prec L_{\eta^+}[\mathcal{Q},\mathcal{T}]$,
 let $\sigma_X : M_X \to X$ denote the uncollapse map of $X$, define the tree $\mathcal{T}_X = \sigma_X^{-1}(\mathcal{T})$ on $\mathcal{Q}$, and note that $\mathcal{T}_X$ is a hull of $\mathcal{T}$
 as witnessed by the map $\sigma_X \restriction \lh(\mathcal{T}_X)$.

 Furthermore, for any two appropriate elementary substructures $X, Y \prec L_{\eta^+}[\mathcal{Q},\mathcal{T}]$ such that $X\subset Y$,
 let $\sigma_{XY} :M_X \to M_Y$ denote the factor map $\sigma_{Y}^{-1}\circ \sigma_X$
 and note that $\mathcal{T}_X$ is a hull of $\mathcal{T}_Y$ as witnessed by the map $\sigma_{XY} \restriction \lh(\mathcal{T}_X)$.
 
 We say that an elementary substructure $X \prec L_{\eta^+}[\mathcal{Q},\mathcal{T}]$ is \emph{stable} if it is appropriate
 and for every appropriate
 elementary substructure $Y \prec L_{\eta^+}[\mathcal{Q},\mathcal{T}]$ such that $X \subset Y$ we have
 \[ \Lambda(\mathcal{T}_X) = \sigma_{XY}^{-1}[\Lambda(\mathcal{T}_Y)]. \]
 Note that an equivalent condition would be $\sigma_X[\Lambda(\mathcal{T}_X)] \subset \sigma_Y[\Lambda(\mathcal{T}_Y)]$
 because distinct cofinal branches are eventually disjoint.
 
 Assume for the moment that there is a stable elementary substructure $X\prec L_{\eta^+}[\mathcal{Q},\mathcal{T}]$.
 Then we can define the branch $b$ of $\mathcal{T}$
 to be the downward closure of the set $\sigma_X[\Lambda(\mathcal{T}_X)]$ in the $\mathcal{T}$-ordering.
 For every appropriate elementary substructure $Y \prec L_{\eta^+}[\mathcal{Q},\mathcal{T}]$ such that $X \subset Y$, we have $\sigma_Y^{-1}[b] = \Lambda(\mathcal{T}_Y)$.
 Moreover, the tree $\mathcal{T}_Y^\frown \sigma_Y^{-1}[b]$ is a hull of $\mathcal{T}^\frown b$.\footnote{In general if $\bar{\mathcal{U}}$ is a hull of an iteration tree $\mathcal{U}$ as witnessed by a map $\sigma : \lh(\bar{\mathcal{U}}) \to \lh(\mathcal{U})$, $c$ is a cofinal branch of $\mathcal{U}$, and $c \cap \operatorname{range}(\sigma)$ is cofinal in $\lh(\mathcal{U})$, then $\bar{\mathcal{U}}^\frown\sigma^{-1}[c]$ is a hull of $\mathcal{U}^\frown c$.} 
 Therefore club many countable hulls of $\mathcal{T}^\frown b$ are by $\Lambda$ and we can argue as in subcase \eqref{item-T-uncountable-cofinality} that every countable hull of $\mathcal{T}^\frown b$ is by $\Lambda$.
 
 So assume toward a contradiction that there is no stable $X$.
 Let $S$ be a tree on $\omega \times \Ord$ that projects to $\operatorname{Code}(\Lambda)$,
 let $f : \omega \to \lh(\mathcal{T})$ be a cofinal map,
 and define
 the model $N' = L[q,\mathcal{T},S,f]$.  (Recall that $q$ is a real coding the premouse $\mathcal{Q}$.)
 Note that the model $N'$ satisfies the statement ``there is no stable $X$'' as well as $V$ does: for any appropriate elementary substructure $X \prec L_{\eta^+}[\mathcal{Q},\mathcal{T}]$ in $N'$, we may use the absoluteness of wellfoundedness of the tree of attempts to find an appropriate elementary substructure $Y \prec L_{\eta^+}[\mathcal{Q},\mathcal{T}]$ such that $X \subset Y$
 but $\Lambda(\mathcal{T}_X) \ne \sigma_{XY}^{-1}[\Lambda(\mathcal{T}_Y)]$. (We may use the tree $S$ to witness values of $\Lambda$.)
 
 Define $\gamma = \omega_1^{N'}$ and note that $\gamma < \omega_1$, just as for the model $N$ in the uncountable cofinality case.
 In the model $N'$ we can build a continuous, $\subset$-increasing sequence $(X_\alpha : \alpha \le \gamma)$ of appropriate elementary substructures of $L_{\eta^+}[\mathcal{Q},\mathcal{T}]$ 
 such that 
 \[\Lambda(\mathcal{T}_\alpha) \ne \sigma_{\alpha,\alpha+1}^{-1}[\Lambda(\mathcal{T}_{\alpha+1})]\] for all $\alpha < \gamma$,
 where we define $\mathcal{T}_\alpha = \mathcal{T}_{X_\alpha}$, $\sigma_\alpha = \sigma_{X_\alpha}$, \emph{etc.}
 
 Define the cofinal branch $b = \Lambda(\mathcal{T}_{\gamma})$ of $\mathcal{T}_{\gamma}$ and
 note that this branch is in the model $N'$ because it can be computed using the tree $S \in N'$.
 For all sufficiently large $\alpha < \gamma$ the intersection $b \cap \sigma_{\alpha,\gamma}[\lh(\mathcal{T}_\alpha)]$ is cofinal in $\lh(\mathcal{T}_\gamma)$,
 which implies that the tree $\mathcal{T}_{\alpha}^\frown \sigma_{\alpha,\gamma}^{-1}[b]$ is a hull of
 $\mathcal{T}_{\gamma}^\frown b$.
 So by hull condensation we have $\sigma_{\alpha,\gamma}^{-1}[b] = \Lambda(\mathcal{T}_{\alpha})$ for all such $\alpha$,
 and by considering such an $\alpha$ and its successor we get $\Lambda(\mathcal{T}_\alpha) = \sigma_{\alpha,\alpha+1}^{-1}[\Lambda(\mathcal{T}_{\alpha+1})]$, a contradiction. 
\end{proof}

\section{$\Omega^*$ IS CONSTRUCTIBLY CLOSED}
\label{omega_closed}

The main theorem of this section is the following.
\begin{theorem}[$\sf{ZF+DC}_\mathbb{R}$]\label{OmegaClosed}
Assume there are no transitive $\sf{AD}^+$ models $M$ containing $\mathbb{R}\cup \rm{OR}$ such that there is a pointclass $\Gamma\subsetneq \powerset(\mathbb{R})^M$ such that $L(\Gamma)\cap\powerset(\mathbb{R})=\Gamma$ and $L(\Gamma)\vDash \sf{AD}_\mathbb{R} + \sf{DC}$. Then $L(\Omega^*)\cap \powerset(\mathbb{R}) = \Omega^*$.
\end{theorem}
\begin{remark}
We note that the smallness assumption in Theorem \ref{OmegaClosed} is stronger than $(\dag)$. It allows for the existence of a minimal model of ``$\sf{AD}_\mathbb{R}+\sf{DC}$" but not much more. The Solovay sequence of the minimal model of ``$\sf{AD}_\mathbb{R}+\sf{DC}$" has length $\omega_1$. We will use $(\dag^+)$ to denote this hypothesis.
\end{remark}
We assume $(\dag^+)$ throughout this section. Suppose the Solovay sequence of $\Omega^*$ is of successor length. Then by Section \ref{CMIInsideLp}, $\Omega^* = \powerset(\mathbb{R})\cap M$, where for some operator $\F$,
\begin{center}
$M = \bigcup\{\M \lhd \rm{Lp}$$^{^\gTheta\F}(\mathbb{R},\rm{Code}(\F))$$ \ | \ \M\vDash \sf{AD}^+ \wedge \M \textrm{ is self-iterable} \}$,\footnote{This means whenever $\M^*$ is countable, transtive and there is an elementary embedding from $\M^*$ into $\M$, then $\M^*$ is $(\omega,\omega_1+1)$-$\F$-iterable.}
\end{center}
and furthermore, Section \ref{CMIInsideLp} also shows that 
\begin{center}
$\powerset(\mathbb{R})\cap M = \powerset(\mathbb{R})\cap L[M]$.
\end{center}
Clearly, this then shows that $\Omega^* = \powerset(\mathbb{R})\cap L(\Omega^*)$.

Suppose now the Solovay sequence of $\Omega^*$ is of limit length. Let $\mathcal{H}$ be the direct limit of all hod pairs $(\Q,\Lambda)\in \Omega^*$ such that $\Lambda$ has branch condensation and is $\Omega^*$-fullness preserving. $\mathcal{H}$ is a union of hod premice and by $(\dag)$ and \cite{ATHM}, $\mathcal{H}$ has ordinal height $\Theta^{\Omega^*}$.\footnote{In fact, the universe of $\mathcal{H}$ is precisely the set of all $A$ bounded subset of $\Theta^{\Omega^*}$ such that $A$ is $OD$ in $L(B,\mathbb{R})$ for some $B\in \Omega^*$.} Let $\lambda$ be the order type of the Solovay sequence of $\Omega^*$; so $\lambda$ is a limit ordinal by the previous sections. By the smallness assumption of the theorem, $\lambda\leq \omega_1$. From now on, we write $\Theta^*$ for $\Theta^{\Omega^*}$ and $\theta^*_\alpha$ for each $\theta^{\Omega^*}_\alpha$ on the Solovay sequence of $\Omega^*$.

The following is the main lemma.
\begin{lemma}[$\sf{ZF} + \sf{DC}_\mathbb{R}$]\label{NotProjecting}
There are no $\M\unlhd L[\mathcal{H}]$ such that $\mathcal{H}\in \M$ and $\rho_\omega(\M)<\Theta^*$.
\end{lemma}
\begin{proof}
Suppose not. Let $\N\unlhd L[\mathcal{H}]$ be least such that $\rho_\omega(\N)<\Theta^*$. Let $B\in \Omega^*$ be of Wadge rank $\theta^*_{n+1}$ where $n<\lambda$ is such that $\rho_\omega(\N) \leq \theta_n^*$ and $\theta^*_n \geq \upsilon$, where $\upsilon$ is the $\N$-cofinality of $\lambda$. Suppose $k$ is the least such that a $\rho_{k+1}(\N)<\Theta^*$; we may assume $\rho_{k+1}(\N)\leq\theta_n^*$. Let $M = L_\gamma(\mathbb{R},B,\N)$, where $\gamma$ is some sufficiently large cardinal so that $L_\gamma(\mathbb{R},B,\N)\vDash \sf{ZF}^-+\sf{DC}$.

For countable $\sigma\prec M$ containing all relevant objects, let $\pi_\sigma: M_\sigma\rightarrow M$ be the transitive uncollapse map whose range is $\sigma$. Such a $\sigma$ exists by $\sf{DC}$ in $L(\mathbb{R},B,\N)$. For each such $\sigma$, let $\pi_\sigma(\mathcal{H}_\sigma,\Theta_\sigma,\lambda_\sigma, \N_\sigma, B_\sigma,\upsilon_\sigma) = (\mathcal{H},\Theta^*,\lambda,\N, B,\upsilon)$. Let $\Sigma_\sigma^- = \oplus_{\alpha<\lambda_\sigma} \Sigma_{\mathcal{H}_\sigma(\alpha)}$. Note that for each $\alpha<\lambda_\sigma$, $\Sigma_{\mathcal{H}_\sigma(\alpha)}$ acts on all countable stacks as it is the pullback of some hod pair $(\R,\Lambda)$ with the property that $\M_\infty(\R,\Lambda) = \mathcal{H}(\pi_\sigma(\alpha))$.

Let $\sigma\prec M$ be such that $\omega_1^{M_\sigma}>n$; this is possible since $n <\lambda\leq \omega_1$. $\Sigma_{\mathcal{H}_\sigma(n+1)}$ is $\Omega^*$-fullness preserving and has branch condensation. This follows from the choice of $B$, which gives that $(\mathcal{H}_\sigma(n+1),\Sigma_{\mathcal{H}_\sigma(n+1)})$ is a tail of some hod pair $(\Q,\Lambda)\in M_\sigma$ such that $\Q$ has $n+1$ Woodin cardinals and $\Lambda$ has branch condensation and is $\Omega^*$-fullness preserving. We let $\Sigma_\sigma^{n}$ be the fragment of $\Sigma^-_\sigma$ for stacks on $\N_\sigma$ above $\delta^{\N_\sigma}_n$. Note that $\Sigma_\sigma^n$ is an iteration strategy of $\N_\sigma$ above $\delta^{\N_\sigma}_n$  since $\Sigma^n_\sigma$-iterations are above $\upsilon_\sigma$, which may be measurable in $\N_\sigma$, and hence does not create new Woodin cardinals. $\Sigma_\sigma^n$ has branch condensation. We then have that $\Sigma_\sigma^n\in \Omega^*$; otherwise, by results in the previous sections, we can show $L(\Sigma^n_\sigma,\mathbb{R})\vDash \sf{AD}^+$ and this contradicts the definition of $\Omega^*$.\footnote{We also have that $\Sigma_\sigma^n$ is the join of countably many sets of reals, each of which is in $\Omega^*$ and hence is Suslin co-Suslin. This implies that $\Sigma_\sigma^n$ is self-scaled.} Also, by \cite[Theorem 3.26]{ATHM}, $\Sigma^n_\sigma$ is $\Gamma=_{\rm{def}} \Gamma(\N_\sigma,\Sigma^n_\sigma)$-fullness preserving.


We then consider the directed system $\mathcal{F}$ of tuples $(\Q,\Lambda)$ where $\Q$ agrees with $\N_\sigma$ up to $\delta_n^{\N_\sigma}$, and $(\Q,\Lambda)$ is Dodd-Jensen equivalent to  $(\mathcal{H}_\sigma,\Sigma_\sigma^{n})$, that is $(\Q,\Lambda)$ and $(\mathcal{H}_\sigma,\Sigma_\sigma^{n})$ coiterate (above $\delta_n^{\N_\sigma}$) to a hod pair $(\R,\Psi)$. $\mathcal{F}$ can be characterized as the directed system of hod pairs $(\Q,\Lambda)$ extending $(\N_\sigma(n),\Sigma_{\N_\sigma(n)})$ such that $\Gamma(\Q,\Lambda)= \Gamma$, $\Lambda$ has branch condensation and is $\Gamma$-fullness preserving. We note that $\F$ is $OD_{\Sigma_{\mathcal{H}_\sigma(n)}}$ in $L(C,\mathbb{R})$ for some $C\in\Omega^*$. We fix such a $C$; so $L(C,\mathbb{R})\vDash \sf{AD}^+ + \sf{SMC}$. Let $A\subseteq \delta^{\N_\sigma}_n$ witness $\rho_{k+1}(\N_\sigma)\leq \delta_n^{\N_\sigma}$. Then $A$ is $OD_{\Sigma_{\mathcal{H}_\sigma(n)}}$ in $L(C,\mathbb{R})$. By $\sf{SMC}$ in $L(C,\mathbb{R})$ and the fact that $\N_\sigma(n+1)$ is $\Omega^*$-full, $A\in \rm{Lp}$$^{\Sigma_{\mathcal{H}_\sigma(n)}}(\N_\sigma|\delta^{\N_\sigma}_n)\in \N_\sigma$. This contradicts the definition of $A$.

\end{proof}

In the following, we write, for $\alpha<\lambda$, $\powerset_{\theta_\alpha}(\mathbb{R})$ for $(\powerset_{\theta_\alpha}(\mathbb{R}))^{\Omega^*}$, $\Sigma_\alpha$ for $\Sigma^{\mathcal{H}}_\alpha$. We also need the following notation: let $(\P,\Sigma)\in\Omega^*$ be a hod pair; let $\M^\sharp_{\P,\Sigma} = \M_\omega^{\Sigma,\sharp}$ be the minimal $\P$-sound, active $\Sigma$-mouse with $\omega$ many Woodin cardinals $\delta_0^{\M_{\P,\Sigma}}<\delta_1^{\M_{\P,\Sigma}} < \dots$, and $\delta_\omega^{\M_{\P,\Sigma}} = $ sup$_i\delta_i^{\M_{\P,\Sigma}}$.\footnote{Sections \ref{OneStep} and \ref{CMIInsideLp} show that $\M^\sharp_{\P,\Sigma}$ exists and its canonical strategy is in $\Omega^*$.} Finally, we let $\M_{\P,\Sigma}=\M_{\omega}^{\Sigma}$ be the corresponding class mouse obtained from $\M^\sharp_{\P,\Sigma}$ by iterating the top extender OR many times. We remind the reader that at this point, we assume that $\lambda$ is a limit ordinal.
\begin{lemma}[$\sf{ZF+DC}_\mathbb{R}$]\label{DMResemblance}
Fix $s\in (\Theta^*)^{<\omega}$ and $\alpha<\lambda$ be such that $s\in (\theta^*_\alpha)^{<\omega}$.  Then for any formula $\psi$, for any hod pair $(\Q,\Lambda)\in\Omega^*$ such that $\Lambda$ is $\Omega^*$-fullness preserving, has branch condensation, and $\Gamma(\Q,\Lambda)=\powerset_{\theta_\alpha}(\mathbb{R})$,
\begin{center}
$L(\Lambda,\mathbb{R}) \vDash \psi[s] \Leftrightarrow \M_{\alpha,\infty}\vDash$ ``the derived model   satisfies $\psi[i^{\Sigma_\alpha}_{\mathcal{H}(\alpha),\infty}(s)]$", \ \ \ $(\textasteriskcentered)$
\end{center}
where $\M_{\alpha,\infty}$ is the direct limit of all iterates of $\M_{\Q,\Lambda}$ below $\delta_0^{\M_{\Q,\Lambda}}$ via its canonical strategy. 
\end{lemma}
\begin{proof}
Fix $s,\psi,\alpha,(\Q,\Lambda)$ as in the statement of the lemma. First we note that $\Sigma_\alpha$ is a tail of $\Lambda$. Let $\P=\M_{\Q,\Lambda}$ and $\Sigma$ be the canonical strategy of $\P$ extending $\Lambda$. Note that for any $\Sigma$-iterate $\P^*$ of $\Sigma$, we can iterate $\P^*$ using $\Sigma$ to some $\P'$ such that $L(\Lambda,\mathbb{R})$ the derived model of $\P'$ at $\delta_{\omega}^{\P'}$.\footnote{This is analogous to the fact that $L(\mathbb{R})$ is the derived model of an iterate of $\M_\omega$.} We may assume also that $s$ is in the range of the direct limit map from $\P$ into $\M_{\alpha,\infty}$.

Suppose the left hand side of the equivalence fails, that is
\begin{center}
$L(\Lambda,\mathbb{R})\vDash \neg \psi[s]$.
\end{center}
Work in $V^{Col(\omega,\mathbb{R})}$, let $\{\P_n,\Sigma_n) \ | \ n<\omega\wedge (\P,\Sigma_n)\in I(\P,\Sigma)\}$ be cofinal in the directed system of $\Sigma$-iterates below $\delta_0^{\P}$; here we take $(\P_0,\Sigma_0)=(\P,\Sigma)$.\footnote{There is an awkward point here. We don't know that $(\P,\Sigma)$ is iterable in $V^{Col(\omega,\mathbb{R})}$; but we can run the argument below inside an $L[T,x]$ where $T$ is a tree projecting some universal $\Gamma$ set $A$ and $\Gamma$ is an inductive-like, scaled pointclass beyond $\powerset_{\theta_\alpha}(\mathbb{R})$ and $x\in\mathbb{R}$ codes $\P$ as well as the reduction of $A$ to $Code(\Sigma)$. We may also assume $(\textasteriskcentered)$ is absolute between $\Omega$ and the model $L[T,x]$. Since $\mathbb{R}\cap L[T,x]$ is countable, we can proceed with the argument below pretending that $V$ is $L[T,x]$.} For $m\leq n<\omega$, let $i_n:\P_n\rightarrow \P_{n+1}$ be the iteration map and $i_{m,n}:\P_m\rightarrow \P_n$, $i_{m,\infty}:\P_m\rightarrow \M_{\alpha,\infty}$ be the natural maps. Set $s_0 = i_{0,\infty}^{-1}(s)$ and let $s_n = i_{0,n}(s_0)$. Let $(\P^\omega_k : k<\omega)$, $(\pi_{k,l}:\P^\omega_k\rightarrow \P^\omega_l : k\leq l<\omega)$ come from the simultaneous $\mathbb{R}$-genericity iteration construction described in \cite[Lemma 6.51]{steel2012hod}. We also let $j_i:\P_i\rightarrow \P_i^\omega$ be the iteration map; here the iterations are above the $s_i$'s, i.e. 
\begin{center}
$j_i(s_i) = s_i$.
\end{center}
By properties of the construction, for $k\leq l<\omega$
\begin{center}
$j_l\circ i_{k,l} = \pi_{k,l}\circ j_k$.
\end{center}
Let $\P^\omega_\omega$ be the direct limit of $\P^\omega_k$ under embeddings $\pi_{k,l}$'s and let $\pi_{i,\omega}:\P^\omega_i\rightarrow \P^\omega_\omega$, $j_\omega:\M_{\alpha,\infty}\rightarrow \P^\omega_\omega$ be the natural maps. Note that $j_\omega(s)=s$.

By our assumptions, for each $i$, 
\begin{center}
$\P^\omega_i \vDash 1 \Vdash$ ``the derived model satisfies $\neg \psi[s] + s = i_{\P^\omega_i,\infty}(s_i)$".
\end{center}
Let $k$ be such that for all $l\geq k$, $\pi_{l,l+1}(s)=s$ ($k$ exists because $\P^\omega_\omega$ is well-founded), and let $s^* = \pi_{k,\omega}(s)$. By elementarity,
\begin{center}
$\P^\omega_\omega \vDash 1 \Vdash$ ``the derived model satisfies $\neg\psi[s^*] + s^* = i_{\P^\omega_\omega,\infty}(s)$".
\end{center}
By elementarity of $j_\omega$ and the fact that $j_\omega(s)=s$, we get
\begin{center}
$\M_{\alpha,\infty} \vDash 1 \Vdash$ ``the derived model satisfies $\neg\psi[i_{\P^\omega_\omega,\infty}(s)]$". 
\end{center}
Contradiction. The other direction is proved similarly.
\end{proof}
\begin{remark}\label{UniformDef}
The right hand side of $(\textasteriskcentered)$ can be defined in $\mathcal{H}$ from $\Sigma_\alpha$ uniformly in $\Sigma_\alpha$. This is because the right hand side of $(\textasteriskcentered)$ is equivalent to the statement: in the derived model of $L[\mathcal{H}]$ at the supremum of its Woodin cardinals, the model $L(\Sigma_\alpha,\mathbb{R}^*) \vDash \psi[i^{\Sigma_\alpha}_{\mathcal{H}(\alpha),\infty}]$, where $\mathbb{R}^*$ is the Col$(\omega,<\Theta)$-symmetric reals. This, in turns, is because we can do an $\mathbb{R}^*$-genericity iterations of $\M_{\alpha,\infty}$ in Col$(\omega,<\Theta)$. 
\end{remark}
Recall from \cite{trang2013hod} the following version of the Vopenka algebra. For each $\alpha<\lambda$, let $\mathbb{P}^*_\alpha$ be the boolean algebra $(\{A\subseteq \powerset(\xi)^n \ | \ n<\omega \wedge \xi<\theta_\alpha\wedge \exists B\in\Omega\ A \textrm{ is } OD^{L(B,\mathbb{R})}\},\subseteq)$; let $\mathbb{P}_\alpha\in \mathcal{H}\cap \powerset(\theta_\alpha)$ be the isomorphic copy of $\mathbb{P}_\alpha^*$. It's clear that for each $\alpha$, $\mathbb{P}_\alpha^*$ and $\mathbb{P}_\alpha$ are $OD$ in $L(\powerset_{\theta_\beta})(\mathbb{R})$ for any $\beta>\alpha$ and the definition is uniform in $\alpha$. Furthermore, for $\alpha<\beta$, there is a natural embedding of $\mathbb{P}_\alpha^*$ into $\mathbb{P}_\beta^*$ (and hence from $\mathbb{P}_\alpha$ into $\mathbb{P}_\beta$) and these embeddings are also $OD$ in $L(\powerset_{\theta_\gamma}(\mathbb{R}))$ for any $\gamma>\beta$ and again, the definition is uniform in $\alpha,\beta$. Let $\mathbb{P}$ be the direct limit of the $\mathbb{P}_\alpha$'s under the natural embeddings. The following corollary of Lemma \ref{DMResemblance} shows that $\mathbb{P}\in L[\mathcal{H}]$. We note that in the corollary below, the language of the structure $L[\mathcal{H}]$ has the predicate for the sequence of strategies $\{\Sigma_\alpha \ | \ \alpha<\lambda\}$.

\begin{corollary}\label{PIsIn}
For each $\alpha < \lambda$, $\mathbb{P}_\alpha$ is definable in $L[\mathcal{H}]$ from $\{\theta_{\alpha+1},\Sigma_{\alpha+1}\}$; the definition is uniform in $\alpha$. Similarly, for $\alpha<\beta$, the natural embedding from $\mathbb{P}_\alpha$ into $\mathbb{P}_\beta$ is definable in $L[\mathcal{H}]$ uniformly in $\{\theta_{\alpha+1},\theta_{\beta+1},\Sigma_{\alpha+1},\Sigma_{\beta+1}\}$. Consequently, $\mathbb{P}\in L[\mathcal{H}]$.
\end{corollary}
\begin{proof}
We just prove the first clause, the proof of the second clause is similar. Fix any $\beta>\alpha$; let $(\Q,\Lambda),(\P,\Sigma), \M_{\beta,\infty}$ be defined as in the proof of Lemma \ref{DMResemblance} but for $\Sigma_\beta$. Note that $\P_\alpha\in \mathcal{H}(\beta)$. By Lemma \ref{DMResemblance},
\begin{center}
$L[\mathcal{H}]\vDash 1\Vdash$ in the derived model, $L(\Sigma_\beta,\mathbb{R}^*)$ satisfies ``$i_{\mathcal{H}(\beta),\infty}(\mathbb{P}_\alpha)$ is the Vopenka algebra at $i_{\mathcal{H}(\beta),\infty}(\theta_\alpha)$". 
\end{center} 
The above gives a uniform definition of $\mathbb{P}_\alpha$ from $\{\theta_\beta,\Sigma_\beta\}$ inside $L[\mathcal{H}]$ for any $\beta>\alpha$.

Clearly, the third clause follows from the first two clauses.
\end{proof}
Using Corollary \ref{PIsIn} and \cite[Theorem 4.3.19]{trang2013hod}, we can conclude that
\begin{itemize}
\item $L[\mathcal{H}](\Omega^*)$ is a symmetric extension of $L[\mathcal{H}]$ via $\mathbb{P}$.
\item $\powerset(\mathbb{R})\cap L[\mathcal{H}](\Omega^*) = \Omega^*$.
\end{itemize}
These, in particular, imply $L(\Omega^*)\cap \powerset(\mathbb{R})=\Omega^*$. This completes the proof of Theorem \ref{OmegaClosed}.

Lemma \ref{NotProjecting} shows that $V_{\Theta^\Omega}\cap L[\mathcal{H}] = |\mathcal{H}|$. In the case $L[\mathcal{H}] \vDash$ ``the set of Woodin cardinals has limit order type", let $M$ be the derived model of $L[\mathcal{H}]$ (at the supremum of $L[\mathcal{H}]$'s Woodin cardinals). Then $M\vDash \sf{AD}_\mathbb{R}$ (cf. \cite[Section 3.3]{ATHM}). This, combined with the result of the previous section, prove Theorem \ref{theorem:ADR-equiconsistency}; Theorem \ref{OmegaClosed} proves something stronger, namely, $\Omega^*$ is constructibly closed. 
\begin{lemma}\label{ADRDC}
If $\sf{DC}$ holds, then cof$(\Theta^*)>\omega$ and $L[\mathcal{H}](\Omega^*) \vDash \sf{AD}_\mathbb{R} + \sf{DC}$.
\end{lemma}
\begin{proof}
Suppose cof$(\Theta^*) = \omega$. Let $M$ be a transitive structure containing $\mathcal{H}^+\cup \Omega^*$, where $\mathcal{H}^+=L[\mathcal{H}]|\gamma$, where $\gamma>\Theta^*$ is a regular cardinal in $L[\mathcal{H}]$. Let $\sigma\prec M$ be countable such that $\sigma$ is cofinal in $\Theta^*$. Now the $\pi_\sigma$-realizable strategy $\Sigma_\sigma$ defined in the proof of Lemma \ref{NotProjecting} acts on $\pi^{-1}_\sigma(\mathcal{H}^+)$. $\Sigma_\sigma$ on stacks below $\Theta_\sigma$ is simply $\Sigma_\sigma^-$ in this case; by replacing $(\N_\sigma,\Sigma_\sigma)$ by an iterate, we may assume $\Sigma_\sigma$ has branch condensation. We can show then that $\Sigma_\sigma\in \Omega^*$ as before. Furthermore, letting $i$ be the direct limit map from $\pi_\sigma^{-1}(\mathcal{H}^+)$ into the direct limit $\M_\infty$ of all of its $\Sigma_\sigma$-iterates in $\Omega^*$, then by elementarity
\begin{center}
$\pi_\sigma \rest \Theta_\sigma = i\rest \Theta_\sigma$.
\end{center}
So $i$ is cofinal in $\Theta^*$. This means $\M_\infty|\delta^{\M_\infty} = \mathcal{H}\notin \Omega^*$. But this contradicts the fact that $\Sigma_\sigma\in \Omega^*$.

The second clause follows immediately from the first and the remarks above.

\end{proof}

Lemma \ref{ADRDC} completes the proof of the following theorems.
\begin{theorem}[$\sf{ZF+DC}_\mathbb{R}$]\label{ADR}
Suppose $\Omega^*=\{A\subseteq \mathbb{R} \ | \ L(A,\mathbb{R})\vDash \sf{AD}^+\}$ and $(\dag^+)$ holds. Suppose $\Omega^*\neq \emptyset$ and for every suitable pair $(\P,\Sigma)$ or hod pair $(\P,\Sigma)$ such that $\Sigma$ has branch condensation and is $\Omega^*$-fullness preserving, $\Sigma\in \Omega^*$. If the Solovay sequence of $\Omega^*$ has limit length, then $\Omega^* = L(\Omega^*,\mathbb{R})\cap \powerset(\mathbb{R})$ and $L(\Omega^*,\mathbb{R})\vDash \sf{AD}_\mathbb{R}$.
\end{theorem}
\begin{theorem}[$\sf{ZF+DC}$]\label{ADR+DC}
Suppose $\Omega^*=\{A\subseteq \mathbb{R} \ | \ L(A,\mathbb{R})\vDash \sf{AD}^+\}$ and $(\dag^+)$ holds. Suppose $\Omega^*\neq\emptyset$ and for every suitable pair $(\P,\Sigma)$ or hod pair $(\P,\Sigma)$ such that $\Sigma$ has branch condensation and is $\Omega^*$-fullness preserving, $\Sigma\in \Omega^*$. If the Solovay sequence of $\Omega^*$ has limit length, then $L(\Omega^*,\mathbb{R})\cap \powerset(\mathbb{R})=\Omega^*$ and $L(\Omega^*,\mathbb{R})\vDash \sf{AD}_\mathbb{R} + \sf{DC}$.
\end{theorem}

The above theorems and results of the previous section complete the proof of Theorems \ref{theorem:ADR-equiconsistency} and \ref{theorem:ADR-DC-equiconsistency}.

\section{FURTHER RESULTS, QUESTIONS, AND OPEN PROBLEMS}

We first mention a few natural questions regarding possible weakenings of the hypotheses of Theorems \ref{theorem:AD-equiconsistency} and \ref{theorem:ADR-DC-equiconsistency}. (In some cases one could also formulate versions with fragments of $\mathsf{DC}$ along the lines of \ref{theorem:ADR-equiconsistency}.)

\begin{question}
What are the consistency strengths of the following theories:
\begin{enumerate}
\item $\sf{ZF} + \sf{DC} + ``\omega_1$ is $\omega_2$-strongly compact''?
\item $\sf{ZF} + \sf{DC} + ``\omega_1$ is $\Theta$-strongly compact''?
\end{enumerate}
Are they equiconsistent with $\mathsf{ZF} + \mathsf{DC} + \mathsf{AD}$
and $\mathsf{ZF} + \mathsf{DC} + \mathsf{AD}_\mathbb{R}$ respectively?
\end{question}

One could try to weaken the compactness hypotheses further:

\begin{question}
What are the consistency strengths of the following theories:
\begin{enumerate}
\item $\sf{ZF} + \sf{DC} + {}$``$\omega_1$ is threadable and $\neg \square_{\omega_1}$''?
\item\label{item:threadable-DC} $\sf{ZF} + \sf{DC} + {}$``every uncountable regular cardinal $\le \Theta$ is threadable''?
\end{enumerate}
Are they equiconsistent with $\mathsf{ZF} + \mathsf{DC} + \mathsf{AD}$
and $\mathsf{ZF} + \mathsf{DC} + \mathsf{AD}_\mathbb{R}$ respectively?
\end{question}

However, it may be overly ambitious at present to seek a positive answer especially in case
\ref{item:threadable-DC}; one could try to answer the following question first:

\begin{question}
What is the consistency strength of the theory
$\sf{ZF} + \sf{DC} + ``\omega_1$ is $\mathbb{R}$-strongly compact'' + ``$\Theta$ is singular or threadable''?
Is it equiconsistent with $\mathsf{ZF} + \mathsf{DC} + \mathsf{AD}_\mathbb{R}$?
\end{question}

We mention a corollary of the proof of Theorem \ref{theorem:AD-equiconsistency}.
\begin{theorem}\label{theorem:Theta_omega_2}
The following theories are equiconsistent:
\begin{enumerate}
\item \label{item:AD}$\sf{ZF+DC+AD}$
\item \label{item:Theta_bigger_omega_2}$\sf{ZF+DC} + \Theta>\omega_2 + ``\omega_1$ is $\mathbb{R}$-strongly compact."
\end{enumerate}
\end{theorem}
\begin{proof}
\eqref{item:AD} $\implies$ \eqref{item:Theta_bigger_omega_2}: As explained above, ``$\omega_1$ is $\mathbb{R}$-strongly compact" is a consequence of the existence of the Turing cone measure, which follows from $\sf{AD}$. $\Theta>\omega_2$ follows from the Moschovakis coding lemma.

\eqref{item:Theta_bigger_omega_2} $\implies$ \eqref{item:AD}: Using results in Section \ref{section:intro}, it's easy to see that the hypothesis of \eqref{item:Theta_bigger_omega_2} implies the hypothesis of \ref{item-not-square-omega-1} of Theorem \ref{theorem:AD-equiconsistency}.
\end{proof}

If we strengthen the hypothesis of \eqref{item:Theta_bigger_omega_2} to ``$\omega_1$ is $\mathbb{R}$-supercompact", then we obtain an equiconsistency with ``there are $\omega^2$ many Woodin cardinals", which is strictly stronger than \eqref{item:AD} (or equivalently \ref{item:Theta_bigger_omega_2}) of Theorem \ref{theorem:Theta_omega_2}. This is a result of Woodin (see \cite{trang2012structureLRmu}). Similarly, if we strengthen the hypothesis in \ref{item-PR-strongly-compact-and-DC} of Theorem \ref{theorem:ADR-DC-equiconsistency} to $\sf{ZF+DC} + ``\omega_1$ is $\powerset(\mathbb{R})$-supercompact"\footnote{We say that $\omega_1$ is $X$-supercompact if there is a countably complete, fine, normal measure $\mu$ on $\powerset_{\omega_1}(X)$. $\mu$ is normal on $\powerset_{\omega_1}(X)$ if whenever $F:\powerset_{\omega_1}(X)\rightarrow \powerset_{\omega_1}(X)$ is such that $\{\sigma \mid F(\sigma)\subseteq \sigma \wedge F(\sigma)\neq\emptyset\}\in\mu$ then there is some $x\in X$ such that the set $\{\sigma \mid x\in F(\sigma)\}\in \mu$.} then one obtains the sharp for a model of $\sf{AD}_\mathbb{R} + \sf{DC}$. To see this, note that from the result of Theorem \ref{theorem:ADR-DC-equiconsistency}, we get a model $L(\Omega^*,\mathbb{R})\vDash \sf{AD}_\mathbb{R}+\sf{DC}$, where $\Omega^*\subseteq \powerset(\mathbb{R})$. Fix a countably complete, fine, normal measure $\mu$ on $\powerset_{\omega_1}(\powerset(\mathbb{R}))$. Then note that by normality, 
\begin{center}
$\forall^*_{\mu}\sigma \ M_\sigma=L(\Omega^*_\sigma,\mathbb{R}_\sigma)\vDash \sf{AD}_\mathbb{R}+\sf{DC}$,
\end{center} 
where we have that $\Omega^* = [\sigma\mapsto \Omega^*_\sigma]_\mu$ and $\mathbb{R}=[\sigma\mapsto\mathbb{R}_\sigma]_\mu$. Now, $\forall^*_\mu \sigma \ (\Omega^*_\sigma,\mathbb{R}_\sigma)^\sharp$ exists; by normality again, the sharp for $L(\Omega^*,\mathbb{R})$ exists.   This demonstrates that the theory $\sf{ZF+DC} + ``\omega_1$ is $\powerset(\mathbb{R})$-supercompact" is strictly stronger than $\sf{ZF+DC} + ``\omega_1$ is $\powerset(\mathbb{R})$-strongly compact".

However, we don't know the exact consistency strength of $\sf{ZF+DC} + ``\omega_1$ is $\powerset(\mathbb{R})$-supercompact".

\begin{question}
What is the exact consistency strength of $\sf{ZF+DC} + ``\omega_1$ is $\powerset(\mathbb{R})$-supercompact"?
\end{question}

We end with the following set of questions.
\begin{question}
What are the consistency strengths of the following theories:
\begin{enumerate}
\item ``$\sf{ZF + DC} + ``\omega_1$ is $\powerset(\powerset(\mathbb{R}))$-strongly compact"? 
\item ``$\sf{ZF + DC} + ``\omega_1$ is $\powerset(\powerset(\mathbb{R}))$-supercompact"? 
\item \label{item:strong_compact} ``$\sf{ZF + DC} + \omega_1$ is strongly compact"?
\item \label{item:super_compact}``$\sf{ZF + DC} + \omega_1$ is supercompact"?
\end{enumerate}
In particular, are the theories \eqref{item:strong_compact} and \eqref{item:super_compact} equiconsistent?
\end{question}

It's worth noting that Woodin (unpublished) has shown the theory ``$\sf{ZF + DC} + \omega_1$ is supercompact" is consistent relative to a proper class of Woodin limits of Woodin cardinals. We hope the techniques in this paper when combined with the theory of hod mice would allow us to make significant progress in answering these questions.

\bibliographystyle{plain}
\bibliography{Rmicebib}
\end{document}